\documentclass[12pt]{amsart}%
\usepackage{amsfonts}
\usepackage{amsmath}
\usepackage{amssymb}
\usepackage{graphicx}%
\makeatletter
\@addtoreset{equation}{section}
\makeatother

\newtheorem{theorem}{Theorem}[section]

\newtheorem{lemma}[theorem]{Lemma}

\newtheorem{proposition}[theorem]{Proposition}
\newtheorem{remark}[theorem]{Remark}

\makeatletter
\def\@makefnmark{}
\makeatother

\begin{document}
\title[Trace Trudinger-Moser and Adams inequalities]{Sharp critical and subcritical trace Trudinger-Moser and Adams inequalities on the upper half spaces}

\author{Lu Chen, Guozhen Lu,  Qiaohua Yang, MaocHun Zhu}
\address{School of Mathematics and Statistics, Beijing Institute of Technology, Beijing 100081, P. R. China}
\email{chenlu5818804@163.com}

\address{Department of Mathematics\\
University of Connecticut\\
Storrs, CT 06269, USA}
\email{guozhen.lu@uconn.edu}

\address{School of Mathematics and Statistics\\
Wuhan University\\
 Wuhan, 430072, P. R. China}
\email{qhyang.math@gmail.com; qhyang.math@whu.edu.cn}

\address{School of Mathematical Sciences, Institute of Applied System Analysis\\
Jiangsu University\\
Zhenjiang, 212013, P. R. China}
\email{zhumaochun2006@126.com}

\thanks{The first author was supported by NNSF (No.11901031); the second author was partly supported by a grant from the Simons Foundation; the third author was  partially supported by NNSF (No.12071353); the fourth author was
partially supported   by NNSF  (No. 12071185).}

\begin{abstract}
    In this paper, we establish the sharp critical and subcritical trace Trudinger-Moser and Adams inequalities on the half spaces and prove the existence of their extremals through the method based on the Fourier rearrangement, harmonic extension and scaling invariance. These  trace Trudinger-Moser and Adams inequalities  can be considered as the borderline case of the Sobolev trace inequalities of first and higher orders. Furthermore, we show the existence of the least energy solutions for a class of bi-harmonic equations with nonlinear Neumann boundary condition associated with the trace Adams inequalities.
\end{abstract}

\maketitle {\small {\bf Keywords:} Trace Trudinger-Moser inequality, Trace Adams inequality, Nonlinear Neumann boundary condition, Harmonic extension, Pohozaev identity, Ground state, Fourier rearrangement. \\

}
\section{Introduction}
The main content of this paper is concerned with the problem of finding optimal trace Trudinger-Moser and Adams inequalities and existence of their extremals in unbounded domain. Sharp Trudinger-Moser inequalities and its high-order form (Adams inequalities) have been a matter of intensive research due to the importance of these inequalities in application to problems in analysis, PDEs, differential geometry, mathematical physics, etc.
\medskip

The Trudinger inequality as the critical case of the Sobolev imbedding was obtained by Trudinger \cite{Tru}, Pohazaev \cite{Pohazaev} and Yudovic \cite{Yudovic}. More precisely,
Trudinger employed the power series expansion to prove that there exists $\alpha>0$ such that
\begin{equation}\label{Tru}
\sup_{\|\nabla u\|_n^n\leq 1, u\in W^{1,n}_0(\Omega)}\int_{\Omega}\exp(\alpha|u|^{\frac{n}{n-1}})dx<\infty,
\end{equation}
where $\Omega\subseteq \mathbb{R}^n$ is a smooth bounded domain and $W^{1,n}_0(\Omega)$ denotes the usual Sobolev space, i.e, the completion of $C_{0}^{\infty}(\Omega)$
with the norm $$\|u\|_{W^{1,n}_0(\Omega)}=\left(\int_{\Omega}|\nabla u|^ndx\right)^{\frac{1}{n}}.$$
Subsequently, Moser \cite{Mo} utilized the technique of the symmetry and rearrangement to give the sharp constants $\alpha_n=nw_{n-1}^{\frac{1}{n-1}}$ of the above inequality. Such an inequality in the sharp form  of constant $\alpha_n$ is now called  the \emph{Trudinger-Moser} inequality. Since the Poly\'{a}-Szeg\"{o} inequality on which the technique of the symmetry and rearrangement depends is not valid on the high-order Sobolev space, this adds much challenge to the research of high-order Trudinger-Moser inequalities, namely Adams type inequalities. Adams employed the method of representation formulas together the method of the rearrangement of convolutions by O'Neil \cite{Oneil} to establish the sharp high-order Trudinger-Moser inequality on any bounded domain. More precisely, he proved the following result.
\vskip0.1cm

\noindent\textbf{Theorem A.} (\cite{AD})
\textit{ Let $\Omega$ be an open and bounded set in $\mathbb{R}^n$. If $m$ is a positive integer less than $n$, then there exists a constant $C_0=C(n,m)>0$ such that for any $u\in W_0^{m,\frac{m}{n}}(\Omega)$ and $\|\nabla^m u\|_{\frac{n}{m}}\leq 1$, then
\begin{equation}\label{adams1}
\frac{1}{|\Omega|}\int_{\Omega}\exp(\beta|u(x)|^{\frac{n}{n-m}})dx\leq C_0,
\end{equation}}
\textit{for all $\beta\leq \beta(n,m)$, where }
\[
\beta(n,m)=
\begin{cases}
\frac{n}{\omega_{n-1}}\big[\frac{\pi^{\frac{n}{2}}2^m\Gamma(\frac{m+1}{2})}{\Gamma(\frac{n-m+1}{2})}
\big]^{\frac{n}{n-m}}, &\text{if $m$ is odd.}
\\
\frac{n}{\omega_{n-1}}\big[\frac{\pi^{\frac{n}{2}}2^m\Gamma(\frac{m}{2})}{\Gamma(\frac{n-m}{2})}
\big]^{\frac{n}{n-m}}, &\text{if $m$ is even.}\\
\end{cases}
\]
Furthermore, the constant $\beta(n,m)$ is best possible in the sense that for any $\beta>\beta(n,m)$, the integral can be made as large as possible.
\vskip 0.2cm

Later, Moser's results for the first order derivatives and Adams' result for the high order derivatives have been extended to the unbounded domains.
  The first order Trudinger-Moser inequality was proved in \cite{Cao}, 
\cite{J.M. do1} and  \cite{Adachi-Tanaka} in the subcritical case,  that is for any $\alpha<\alpha_{n}$, there exists a constant $C=C(\alpha, n)$ such that

 \begin{equation}\label{LR}
\sup_{\|\nabla u\|^n_{n}\leq 1}\int_{\mathbb{R}^n}\Phi (\alpha|u|^{\frac{n}{n-1}})dx\leq C \|u\|_n^n,
\end{equation}
 where $\Phi\left(  t\right)  =e^{t}-\underset{j=0}{\overset{n-2}{\sum }}t^j$.

 Later, it is showed in  \cite{LR,R}  that the exponent
$\alpha_{n}$ becomes admissible if the Dirichlet norm
$(\int_{\mathbb{R}^n}\left\vert \nabla u\right\vert ^{n}dx)^{\frac{1}{n}}$ is replaced by
Sobolev norm $\|u\|_{W^{1,n}(\mathbb{R}^n)}=(\int_{\mathbb{R}^n}\left(  \left\vert u\right\vert
^{n}+\left\vert \nabla u\right\vert ^{n}\right)  dx)^{\frac{1}{n}}$, more
precisely, they proved the following critical Trudinger-Moser inequality
\begin{equation}\label{LR}
\sup_{\|u\|_{W^{1,n}(\mathbb{R}^n)}\leq 1}\int_{\mathbb{R}^n}\Phi (\alpha_n|u|^{\frac{n}{n-1}})dx\leq C_n<\infty.
\end{equation}
  All the proofs of both critical and subcritical Trudinger-Moser inequalities in the literature rely on rearrangement argument and the Poly\'{a}-Szeg\"{o} inequality.  The authors of \cite{LL5}, \cite{LL2}   used a symmetrization-free approach to give a simple proof for both critical and subcritical sharp Trudinger-Moser inequalities in $W^{1,n}(\mathbb{R}^n)$. It should be noted that this approach is surprisingly simple and can be easily applied to other settings where symmetrization argument does not work. Furthermore, they also develop this new approach to establish the global Trudinger-Moser inequalities on the entire Heisenberg group from the local one derived in \cite{CL} and fractional Adams inequalities in $W^{s,\frac{n}{s}}(\mathbb{R}^n)$ ($0<s<n$) (see \cite{LL5,LL2, LamLuTang}). For more applications of the symmetrization-free approach, please also see \cite{LaZ, LiLu0, LLZ, ZC}. Moreover, the critical and subcritical Trudinger-Moser inequalities are equivalent as shown in \cite{LaZ}.
\medskip

As far as the existence of extremal functions of Trudinger-Moser inequality on bounded domains, the first breakthrough was due to Carleson and Chang \cite{CC} who established the existence on balls. (See also \cite{Flu}, \cite{lin} for smooth domains and the methods of using expansion formula of Dirichlet integrals \cite{MM}, \cite{ManciniMartinazzi}.)

Since the main purpose of this paper is to establish the trace Trudinger-Moser and Adams type inequalities 
on unbounded domains, we will only briefly review of 
 the trace Trudinger-Moser and Adams inequalities on bounded domains. 
Let $\Omega$ be bounded domain in $\mathbb{R}^n$ with smooth boundary and define the  $W^{1,2}(\Omega)$ to be the Sobolev space, equipped with the norm
$$\|u\|_{W^{1,2}(\Omega)}=\big(\int_{\Omega}|(\nabla u|^2+|u|^2)dx\big)^{\frac{1}{2}}.$$
As is well known, classical Sobolev trace embedding on bounded domain asserts that $W^{1,2}(\Omega)\subset L^{q}(\partial\Omega)$ for $1\leq q\leq \frac{2(n-1)}{n-2}$ and $n>3$. In the borderline case $n=2$, the
$W^{1,2}(\Omega)\subset L^q(\partial\Omega)$ for $1\leq q<\infty$, but $W^{1,2}(\Omega)\nsubseteq L^\infty(\partial\Omega)$. To fill this gap, the authors of \cite{AD0, Adimurthi,  Cia, Cherrier,  LiLiu} studied the trace Trudinger-Moser inequalities and established the following
\vskip 0.1cm

\noindent\textbf{Theorem B.}
Let $\Omega$ be a sufficiently smooth bounded domain in $\mathbb{R}
^{2}$. For any $\alpha\leq \pi$,  then there exists a positive constant $C_n$ such that
\begin{equation}\label{Trace-Trudinger}
\int_{\partial\Omega}\exp(\alpha
|u|^{2}) d\sigma\leq C_n,
\end{equation}
for any $u\in C^{\infty}(
\Omega)$ with $\int_{\Omega} (|\nabla u|
^{2}+|u|^2)dx\leq1$.
\medskip

We should mention that the critical trace inequality of Trudinger-Moser type on balls in $\mathbb{R}^2$ for holomorphic functions with mean value zero was first established by Chang and Marshall \cite{ChangMarshall}. Recently, the authors  in \cite{LiLu} establish the Chang-Marshall type inequality  to sufficiently  smooth domains $\Omega$ in $\mathbb{R}^n$ for all $n\ge 2$ for any Sobolev function $u$ with mean value $\int_{\Omega}udx=0$. We should also note that the author of \cite{Cia} established Trudinger-Moser trace type inequalities on lower dimensional subsets
of $\Omega$ with respect to measures supported on the subsets.

 In a half space, a particular unbounded domain with the smooth boundary, sharp geometrical inequalities such as trace Sobolev inequalities on this domain have attracted much attention due to their importance in geometrical analysis and PDEs (see \cite{Ache, Be, Es,NNP,Yang}). We recall some classical results that involve the sharp trace Sobolev inequalities on the half space. Sharp trace Sobolev inequalities for first order derivative of Escobar \cite{Es} assert that
\vskip0.1cm

\noindent\textbf{Theorem C.}(\cite{Es})
Let $n\geq 3$. Then for any $U\in W^{1,2}(\overline{\mathbb{R}^n_{+}})$, there holds
\begin{equation}
2\frac{\Gamma(\frac{n}{2})}{\Gamma(\frac{n-1}{2})}w_{n-1}^{\frac{1}{n-1}}\big(\int_{\partial\mathbb{R}^n_{+}}|U(x, 0)^{\frac{2(n-1)}{n-2}}dx\big)^{\frac{n-2}{n-1}}
\leq \int_{\mathbb{R}^{n}_{+}}|\nabla U(x,y)|^2dxdy,
\end{equation}
where $(x,y)\in \mathbb{R}^{n-1}\times \mathbb{R}^{+}$.
Furthermore, the equality case is achieved by a harmonic extension of a function of the form
$c(1+|x-x_0|^2)^{-\frac{n-2}{2}}$, where $c$ is a constant, $x\in \mathbb{R}^{n-1}$, $x_0$ is some fixed point in $\mathbb{R}^{n-1}$.
\vskip 0.1cm
By conformal map, this inequality is in fact equivalent to the trace Sobolev inequalities on the ball (see Beckner \cite{Be}),
\begin{equation}\label{trace so}
\frac{n-2}{2}w_{n-1}^{\frac{1}{n-1}}\big(\int_{\mathbb{S}^{n-1}}|f|^{\frac{2(n-1)}{n-2}}d\sigma\big)^{\frac{n-2}{n-1}}\leq \int_{\mathbb{B}^{n}}|\nabla v|^2dx+\frac{n-2}{2}\int_{\mathbb{S}^{n-1}}|f|^2d\sigma,\end{equation}
where $v$ is a smooth extension of $f$ to $\mathbb{B}^{n}$ and $d\sigma$ is the surface measure on $\mathbb{S}^{n-1}$.

 Motivated by the above works, we are concerned with the borderline case of trace Sobolev inequality, that is the sharp trace Trudinger-Moser inequalities in $W^{1,2}(\mathbb{R}^{2}_{+})$. We first establish the subcritical and critical trace-type Trudinger-Moser inequalities on the half space. Our results read as
\begin{theorem}\label{addthm1}
For $0<\beta<\pi$, then there exists a positive constant $C$ such that for all functions $U\in W^{1,2}(\overline{\mathbb{R}^2_{+}})$  with $\int_{\mathbb{R}^2_{+}}|\nabla U(x,y)|^2dxdy\leq 1$, the following inequality holds
\begin{equation}\label{addsub}
\int_{\partial\mathbb{R}^{2}_{+}}(\exp(\beta|U(x,0)|^{2})-1)dx
\leq C \int_{\partial\mathbb{R}^{2}_{+}}|U(x,0)|^{2}dx.
\end{equation}
 Furthermore, the smallest $C$ in \eqref{addsub} can be attained by some $U$, which is a harmonic extension of
some even function $u\in W^{\frac{1}{2},2}(\mathbb{R})$.
Moreover, the constant $\pi$ is sharp in the sense that the inequality \eqref{addsub} fails if the constant $\beta$ is replaced by any $\beta\geq \pi$.
\end{theorem}

\begin{theorem}\label{addthm2}
There exists a positive constant $C$ such that for all functions $U\in W^{1,2}(\overline{\mathbb{R}^2_{+}})$ with $$\int_{\mathbb{R}^2_{+}}|\nabla U(x,y)|^2dxdy+\int_{\partial\mathbb{R}^2_{+}}|U(x,0)|^2dx\leq 1,$$ the following inequality holds
\begin{equation}\label{addcri}
\int_{\partial\mathbb{R}^{2}_{+}}\exp(\pi|U(x,0)|^{2})dx
\leq C.
\end{equation}
Furthermore, the smallest $C$ in \eqref{addcri} can be attained by some function $U$, which is a harmonic extension of
some radial function $u\in W^{\frac{1}{2},2}(\mathbb{R})$. Moreover, the constant $\pi$ is sharp in the sense that the inequality \eqref{addcri} fails if the constant $\pi$ is replaced by any $\beta>\pi$.
\end{theorem}
\begin{remark}
The difference between Theorem \ref{addthm1} and Theorem \ref{addthm2} is that
we assume only $\int_{\mathbb{R}^2_{+}}|\nabla U(x,y)|^2dxdy\leq 1$ in Theorem \ref{addthm1} and assume $$\int_{\mathbb{R}^2_{+}}|\nabla U(x,y)|^2dxdy+\int_{\partial\mathbb{R}^2_{+}}|U(x,0)|^2dx\leq 1$$
in Theorem \ref{addthm2}. This results in a {\bf subcritical} trace Trudinger-Moser inequality in Theorem \ref{addthm1} and
a {\bf critical} trace Trudinger-Moser inequality in Theorem \ref{addthm2}.
\end{remark}

Recently, Ache and Chang \cite{Ache} established the second order sharp trace Sobolev-inequalities \eqref{trace so} on the ball:
\vskip0.1cm

\noindent\textbf{Theorem D.}
Let $u\in C^{\infty}(\mathbb{S}^{n-1})$ with $n>4$. Suppose that $v$ is a smooth extension of $u$ to the unit ball $\mathbb{B}^{n}$ which satisfies
the Neumann boundary condition
$$\frac{\partial v}{\partial n}|_{\mathbb{S}^{n-1}}=-\frac{n-4}{2}u.$$
Then  the following inequality holds
\begin{equation}\label{sobolev on the ball}\begin{split}
&2\frac{\Gamma(\frac{n+2}{2})}{\Gamma(\frac{n-4}{2})}w_{n-1}^{\frac{3}{n-1}}\left(\int_{\mathbb{S}^{n-1}}|u|^{\frac{2(n-1)}{n-4}}d\sigma\right)^{\frac{n-4}{n-1}}\\
&\ \ \leq \int_{\mathbb{B}^{n}}|\Delta v|^2dx+2\int_{\mathbb{S}^{n-1}}|\widetilde{\nabla} u|^2d\sigma+\frac{n(n-4)}{2}\int_{\mathbb{S}^{n-1}}|u|^2\sigma,
\end{split}\end{equation}
where $\Delta v$ is the Laplacian of $v$ with respect to the Euclidean metric, and $\widetilde{\nabla} $ is the  gradient on the sphere.  
\vskip 0.1cm

We refer the reader to  extensions to higher order derivatives  trace Sobolev inequalities on half spaces in \cite{Yang1}, \cite{Yang} and \cite{Case}.

By the Mobius transform, inequality \eqref{sobolev on the ball} is equivalent to the following trace Sobolev inequality on the upper half space $\mathbb{R}^{n}_{+}$.
\vskip0.1cm

\noindent\textbf{Theorem E} (\cite{NNP})
Let $n\geq 5$. Then for any $U\in W^{2,2}(\overline{\mathbb{R}^n_{+}})$ satisfying the Neumann boundary condition $\partial_{y}U(x,y)|_{y=0}=0$, we have
\begin{equation}\label{sobolev on the half space}
2\frac{\Gamma(\frac{n+2}{2})}{\Gamma(\frac{n-4}{2})}w_{n}^{\frac{3}{n-1}}(\int_{\partial\mathbb{R}^n_{+}}|U(x,0)^{\frac{2(n-1)}{n-4}}dx)^{\frac{n-4}{n-1}}
\leq \int_{\mathbb{R}^{n}_{+}}|\Delta U(x,y)|^2dxdy.
\end{equation}
Furthermore, the equality  is achieved by a bi-harmonic extension of a function of the form
$c(1+|x-x_0|^2)^{-\frac{n-4}{2}}$, where $c$ is a constant, $x\in \mathbb{R}^{n-1}$, $x_0$ is some fixed point in $\mathbb{R}^{n-1}$.
\vskip0.1cm

In this paper,
we first establish the borderline case of the inequality \eqref{sobolev on the half space} when $n=4$. We employ the method based on the sharp Fourier rearrangement principle and bi-harmonic extension to establish the sharp trace Adams inequalities on the half space and the existence of their extremals. Both subcritical and critical trace Adams inequalities are proved for functions in $W^{2,2}(\overline{\mathbb{R}^4_{+}})$.  
\vskip0.1cm

Our first result in the second order trace Adams inequality is the following subcritical inequality.

\begin{theorem}\label{thm1}
 Let $0<\beta<12\pi^2$. Then there exists a positive constant $C$ such that for all functions $U\in W^{2,2}(\overline{\mathbb{R}^4_{+}})$ satisfying the Neumann boundary condition $\partial_{y}U(x,y)|_{y=0}=0$ with $\int_{\mathbb{R}^4_{+}}|\Delta U(x,y)|^2dxdy\leq 1$, the following inequality holds:
\begin{equation}\label{sub}
\int_{\partial\mathbb{R}^{4}_{+}}(\exp(\beta|U(x,0)|^{2})-1)dx
\leq C \int_{\partial\mathbb{R}^{4}_{+}}|U(x,0)|^{2}dx.
\end{equation}
  Furthermore, the smallest $C$ in \eqref{sub} can be attained by some function $U$, which is the bi-harmonic extension of
some radial function $u\in W^{\frac{3}{2},2}(\mathbb{R}^3)$. Moreover, the constant $12\pi^2$ is sharp in the sense that the inequality fails if the constant $\beta$ is replaced by any $\beta\geq 12\pi^2$.
\end{theorem}

Next, the following critical trace Adams inequality also holds.

\begin{theorem}\label{thm2}
There exists a positive constant $C$ such that for all functions $U\in W^{2,2}(\overline{\mathbb{R}^4_{+}})$ satisfying the Neumann boundary condition $\partial_{y}U(x,y)|_{y=0}=0$ with $$\int_{\mathbb{R}^4_{+}}|\Delta U(x,y)|^2dxdy+\int_{\partial\mathbb{R}^4_{+}}|U(x,0)|^2dx\leq 1,$$ the following inequality holds
\begin{equation}\label{cri}
\int_{\partial\mathbb{R}^{4}_{+}}\exp(12\pi^2|U(x,0)|^{2}\big)dx
\leq C.
\end{equation}
Moreover, the constant $12\pi^2$ is sharp in the sense that the inequality fails to hold uniformly for all $U\in W^{2,2}(\overline{\mathbb{R}^4_{+}})$ if the constant $12\pi^2$ is replaced by any $\beta>12\pi^2$.
\end{theorem}

The second order Adams inequality with the exact growth was established    in $W^{2,2}(\mathbb{R}^4)$ in \cite{MS} and then was extended to $W^{2,n}(\mathbb{R}^n)$ in \cite{LTZ} (see also the first order case in \cite{IMN}, and  \cite{MS3} and singular inequalities and under different norms with exact growth in \cite{LamLu-exact} and \cite{LLZhang}). A natural question is whether there exists trace Adams inequality with the exact growth in $W^{2,2}(\mathbb{R}^{4}_{+})$. In this paper, we establish the following
\begin{theorem}\label{adam-exact}
There exists a positive constant $C$ such that for all functions $U\in W^{2,2}(\overline{\mathbb{R}^4_{+}})$ satisfying the Neumann boundary condition $\partial_{y}U(x,y)|_{y=0}=0$ with $\int_{\mathbb{R}^4_{+}}|\Delta U(x,y)|^2dxdy=1$, the following inequality holds
\begin{equation}\label{addadmexa}
\int_{\partial\mathbb{R}^{4}_{+}}\frac{(\exp(12\pi^2|U(x,0)|^{2})-1)}{(1+|U(x,0)|)^2}dx
\leq C \int_{\partial\mathbb{R}^{4}_{+}}|U(x,0)|^{2}dx.
\end{equation}
Moreover, this inequality fails to hold uniformly for all $U\in W^{2,2}(\overline{\mathbb{R}^4_{+}})$ if the power $2$ in the denominator of the left hand side is replaced by any $p<2$.

\end{theorem}
\begin{remark}
The proof of Theorem \ref{adam-exact} is based on bi-harmonic extension and fractional Adams inequalities with the exact growth in $W^{\frac{n}{2},2}(\mathbb{R}^n)$ for $n=3$. Adams inequalities with the exact growth in $W^{\frac{n}{2},2}(\mathbb{R}^n)$ for even $n$ has been established  in \cite{MS1}(see also \cite{N}). However, the validity of this inequality for odd $n$ still remain open. Using 
the fractional Hardy-Rellich inequalities established  by Beckner in \cite{Be1} together with 
the method of the reduction of orders (see also  \cite{N}), we present a simple proof for the fractional Trudinger-Moser-Adams inequality with the exact growth in $W^{\frac{n}{2},2}(\mathbb{R}^n)$ for all $n\geq 2$.
\end{remark}

\begin{theorem}\label{fratru}
There exists a positive constant $C$ such that for all functions $u\in W^{\frac{n}{2},2}(\mathbb{R}^n)$ with $\int_{\mathbb{R}^n}|(-\Delta)^{\frac{n}{4}}u|^2dx=1$, the following inequality holds
\begin{equation}\label{addadmexa}
\int_{\mathbb{R}^{n}}\frac{(\exp(\beta(n,\frac{n}{2})|u(x)|^{2})-1)}{(1+|u(x)|)^2}dx
\leq C \int_{\mathbb{R}^{n}}|u(x)|^{2}dx.
\end{equation}
Moreover, this inequality fails if the power $2$ in the denominator is replaced by any $p<2$.

\end{theorem}

Furthermore, we also establish the trace Adams inequalities on $W^{m,2}(\overline{\mathbb{R}^{2m}_{+}})$ $(m> 2)$. The proof is based on the following technical lemma whose proof will be given in Section \ref{S5}.

\begin{lemma}\label{lem1.1}
Given a function $u\in W^{m-\frac{1}{2},2}(\mathbb{R}^{2m-1})$.  If we assume that the function $U\in W^{m,2}(\overline{\mathbb{R}^{2m}_+})$
 satisfying
 $$(-\Delta)^{m}U=0$$
 on the upper half space $\mathbb{R}^{2m}_{+}$
and the Neumann boundary conditions
\begin{equation}\label{boudc1}\left.\Delta^{k}U(x,y)\right|_{y=0}=\frac{\Gamma(m)\Gamma(m-\frac{1}{2}-k)}{\Gamma(m-1-k)\Gamma(m-\frac{1}{2})}\Delta^{k}_{x}u(x),\;\;0\leq k\leq\left[\frac{m-1}{2}\right]\end{equation}
and
\begin{equation}\label{boudc2}\left.\partial_{y}\Delta^{k}U(x,y)\right|_{y=0}=0,\;\;0\leq k\leq\left[\frac{m-2}{2}\right],\end{equation}
then we have the following identity
$$\int_{\mathbb{R}^{2m}_{+}}|\nabla^{m}U(x,y)|^{2}dxdy=\frac{\Gamma(m)\Gamma(\frac{1}{2})}{\Gamma(m-\frac{1}{2})}
\int_{\mathbb{R}^{2m-1}}|(-\Delta)^{\frac{2m-1}{4}}u|^{2}dx.$$
\end{lemma}

Once we have proved  Lemma \ref{lem1.1}, we can establish the following both subcritical and critical trace Adams inequalities on $W^{m,2}(\overline{\mathbb{R}^{2m}_{+}})$ $(m> 2)$ by a similar argument as for Theorems \ref{thm1}, \ref{thm2} and \ref{adam-exact} on $W^{2,2}(\overline{\mathbb{R}^{4}_{+}})$.  

\begin{theorem}\label{thm1.1}
 Let $0<\beta<\frac{\Gamma(m)\Gamma(\frac{1}{2})}{\Gamma(m-\frac{1}{2})}\beta(2m-1,\frac{2m-1}{2})=\widetilde{\beta_m}$. Then there exists a positive constant $C$ such that for all functions $U\in W^{m,2}(\overline{\mathbb{R}^{2m}_{+}})$ satisfying  the Neumann boundary conditions \eqref{boudc1} and  \eqref{boudc2} with $\int_{\mathbb{R}^{2m}_{+}}|\nabla^{m}U|^{2}dxdy\leq 1$, the following inequality holds:
\begin{equation}\label{sub-1}
\int_{\partial\mathbb{R}^{2m}_{+}}(\exp(\beta|U(x,0)|^{2})-1)dx
\leq C \int_{\partial\mathbb{R}^{2m}_{+}}|U(x,0)|^{2}dx.
\end{equation}
Moreover, the constant $\widetilde{\beta_m}$ is sharp in the sense that the inequality fails to hold uniformly if the constant $\beta$ is replaced by any $\beta\geq \widetilde{\beta_m}$. Furthermore, equality in \eqref{sub} holds if and only if $U$ is a $m$-harmonic extension of
some radial function $u$ in $W^{\frac{2m-1}{2},2}(\mathbb{R}^{2m-1})$.
\end{theorem}

\begin{theorem}\label{thm2.1}
There exists a positive constant $C$ such that for all functions $U\in W^{m,2}(\overline{\mathbb{R}^{2m}_{+}})$ satisfying the Neumann boundary conditions \eqref{boudc1} and \eqref{boudc2} with $$\int_{\mathbb{R}^{2m}_{+}}|\nabla^{m}U(x,y)|^{2}dxdy+\int_{\partial\mathbb{R}^{2m}_{+}}|U(x,0)|^2dx\leq 1,$$ the following inequality holds
\begin{equation}\label{cri-1}
\int_{\partial\mathbb{R}^{2m}_{+}}(\exp(\widetilde{\beta_m}|U(x,0)|^{2}\big)-1)dx
\leq C.
\end{equation}
Moreover, the constant $\widetilde{\beta_m}$ is sharp in the sense that the inequality fails to hold uniformly if the constant $\widetilde{\beta_m}$ is replaced by any $\beta>\widetilde{\beta_m}$.
\end{theorem}

\begin{theorem}\label{highadm-exact}
There exists a positive constant $C$ such that for all functions $U\in W^{m,2}(\overline{\mathbb{R}^{2m}_{+}})$ satisfying the Neumann boundary condition \eqref{boudc1} and \eqref{boudc2} with $$\int_{\mathbb{R}^{2m}_{+}}|\nabla^{m}U(x,y)|^{2}dxdy=1,$$ the following inequality holds
\begin{equation}\label{highadmexa}
\int_{\partial\mathbb{R}^{2m}_{+}}\frac{(\exp(\widetilde{\beta_m}|U(x,0)|^{2})-1)}{(1+|U(x,0)|)^2}dx
\leq C \int_{\partial\mathbb{R}^{2m}_{+}}|U(x,0)|^{2}dx.
\end{equation}
Moreover, this inequality fails if the power $2$ in the denominator on the left hand side is replaced by any $p<2$.
\end{theorem}

Finally, we are also concerned with the ground-state solutions to the equation associated with the trace Adams inequality in $W^{m,2}(\mathbb{R}^{2m}_{+})$. For simplicity, we only consider the case $m=2$. By the Euler-Lagrange multiplier theorem, extremals of the trace Adams supremum
\[
 \underset{\int_{\mathbb{R}^4_{+}}|\Delta U(x,y)|^2dxdy+\int_{\partial\mathbb{R}^4_{+}}|U(x,y)|^2dx\leq 1}{\sup}\int_{%
\partial \mathbb{R}^4_{+}}\left(  \exp\left( 12\pi^2\left\vert U\right\vert ^{2}-1\right)
\right)  dx
\]
must satisfy the following Euler-Lagrange equation with nonlinear Neumann boundary conditions: there exists some $\lambda$ such that

\begin{align}\label{Pro}
\begin{cases}
\Delta^{2}U(x,y)=0$ for $(x,y)\in \mathbb{R}^3\times \mathbb{R}^{+}, \\
(-\frac{\partial(\Delta U)}{\partial y}+U(x,y))|_{y=0}=\lambda U(x,y)\exp(12 \pi^2 U^{2}(x,y))|_{y=0}, \\
\frac{\partial U}{\partial y}|_{y=0}=0.
\end{cases}
\end{align}
We are interested in the existence of ground-state solutions to equation \eqref{Pro} with some fixed $\lambda$. Define  $$E:=\{U\in W^{2,2}(\overline{\mathbb{R}^4_{+}})\cap L^{2}(\partial\mathbb{R}^4_{+}):\ \partial_yU(x,y)|_{y=0}=0\},$$  the corresponding functional of equation \eqref{Pro} is

\begin{equation}\label{functional}
I_{\lambda}(U)=\frac{1}{2}\int_{\mathbb{R}^4_{+}}|\Delta U|^2dxdy+\frac{1}{2}\int_{\partial\mathbb{R}^4_{+}}|U|^2dx
 -\frac{\lambda}{24\pi^2}\int_{\partial\mathbb{R}^4_{+}}\big(\exp(12\pi^2|U|^2)-1\big)dx.\end{equation}
It is easy to check that $I_\lambda\in C^1(E,\mathbb{R})$, and
\begin{equation}\begin{split}
I_{\lambda}'(U)V&=\int_{\mathbb{R}^4_{+}}\Delta U(x,y)\Delta V(x,y)dxdy+\int_{\partial\mathbb{R}^4_{+}}U(x,0)V(x,0)dx\\
&\ \ -\lambda\int_{\partial\mathbb{R}^4_{+}}\exp(12\pi^2|U(x,0)|^2)U(x,0)V(x,0)dx,\ \  U,V\in E.
\end{split}\end{equation}
Define
$$M_{\lambda}=\inf \{I^\prime_{\lambda}(U)=0\ |U\in E\}.$$
By using the sharp trace Adams inequalities \eqref{cri} and variational method, we can prove the following result.
\begin{theorem}\label{thm3}
For any $0<\lambda<1$, there exists $V\in E$ such that $I_{\lambda}(V)=M_{\lambda}$, that is to say that the problem (\ref{Pro})   has a ground state solution provided  $0<\lambda<1$.
\end{theorem}

This paper is organized as follows. In Section 2, we establish the sharp subcritical and critical trace Trudinger-Moser inequalities in $\mathbb{R}^2$ and existence of their extremal functions based on the harmonic extension. Section 3 is devoted to the proof of the sharp trace Adams inequalities and the existence of the extremals for subcritical trace Adams inequalities in $\mathbb{R}^4$. In Section 4, we are concerned with the existence of a ground state solution to a  class of problems with the nonlinear Neumann boundary condition on the half space by using the trace Adams inequality, sharp Fourier rearrangement principle and  variational method. Section 5 is devoted to the proof of Lemma \ref{lem1.1} which is the main technical
part of establishing Theorems \ref{thm1.1}, \ref{thm2.1} and \ref{highadm-exact} in the case $m>2$ in $\mathbb{R}^{2m}$ by adapting the same argument of proving Theorems \ref{thm1} and \ref{thm2} in the case $m=2$. In Section 6, we give the proof of the fractional Adams inequalities with the exact growth in $W^{\frac{n}{2},2}(\mathbb{R}^n)$, which is needed in proving Theorem \ref{adam-exact} and (\ref{highadmexa}).

 \section{Trace Trudinger-Moser inequalities on the half space $\mathbb{R}^{2}_{+}$}
In this section, we consider the sharp subcritical and critical trace Trudinger-Moser inequalities and the existence of their extremal functions. The method is based on the harmonic extension, sharp subcritical fractional Adams inequalities in $W^{\frac{1}{2},2}(\mathbb{R})$ and existence of their extremals. First, we introduce some known results about  the fractional Adams inequalities and the harmonic extension.
\begin{lemma}\label{addlem1}(\cite{Fon})
For $0<\beta<\pi$, then there exists a positive constant $C$ such that for all functions $u \in W^{\frac{1}{2},2}(\mathbb{R})$ with $\|(-\Delta)^{\frac{1}{4}} u\|_{2}=1$, the following inequality holds.
\begin{equation}\label{addintq1}
\int_{\mathbb{R}}(\exp(\beta|u|^{2})-1)dx
\leq C\int_{\mathbb{R}}|u|^{2}dx,
\end{equation}
 Moreover, the constant $\pi$ is sharp in the sense that the inequality fails if the constant $\beta$ is replaced by any $\beta\geq\pi$.
\end{lemma}

\begin{lemma} \label{addlem2}
There exists extremals for the above subcritical fractional Adams inequality \eqref{addintq1}. Furthermore, all the extremals of inequality \eqref{addintq1}
must be radially symmetric with respect to some point $x_0 \in \mathbb{R}$.
\end{lemma}
\begin{remark}
The proof can be found in \cite{CLZ}. One can also use the technique of proof of Lemma \ref{lem2} to prove Lemma \ref{addlem2}, then we omit the details.
\end{remark}
\begin{lemma}\label{addlem3}(\cite{Rudin},\cite{Siegel})
Given a function $u\in W^{\frac{1}{2},2}(\mathbb{R})$.  If we assume that function $U\in W^{2,2}(\overline{\mathbb{R}^{2}_{+}})$ satisfying
$$-\Delta U(x,y)=0$$
on the upper half space $\mathbb{R}^{2}_{+}$ and the boundary condition
$$U(x,y)|_{y=0}=u(x),$$ then we have the following identity
\begin{equation}\label{identity}
\int_{\mathbb{R}^2_{+}}|\nabla U(x,y)|^2dxdy=\int_{\mathbb{R}}|(-\Delta)^{\frac{1}{4}}u|^2dx.
\end{equation}
\end{lemma}

Now, we are in the position to prove Theorem \ref{addthm1}.
\medskip

\emph{Proof of Theorem \ref{addthm1}:}
For any $U\in W^{1,2}(\overline{\mathbb{R}^2_{+}})$, we define $V(x,y)$ as the harmonic extension of $U(x,0)$, that is to say $V(x,y)$ satisfies that
$$-\Delta V(x,y)=0$$
on the upper half space $\mathbb{R}^{2}_{+}$ and the boundary condition
$$V(x,y)|_{y=0}=U(x,y)|_{y=0}.$$
Through Green's representation formula, we can write $$V(x,y)=\frac{1}{\pi}\int_{\mathbb{R}}\frac{y}{(|x-\widetilde{x}^2|+y^2)}U(\widetilde{x},0)d\widetilde{x}.$$
According to the Dirichlet principle, we know that
\begin{equation}\label{Dirichletp}
\int_{\mathbb{R}^2_{+}}|\nabla V(x,y)|^2dxdy\leq \int_{\mathbb{R}^2_{+}}|\nabla U(x,y)|^2dxdy.
\end{equation}
Combining (\ref{Dirichletp}) and (\ref{identity}) and Lemma \ref{addlem1}, we obtain
\begin{equation}\begin{split}
\int_{\partial\mathbb{R}^{2}_{+}}(\exp(\beta|U(x,0)|^{2})-1)dx&=\int_{\partial\mathbb{R}^{2}_{+}}(\exp(\beta|V(x,0)|^{2})-1)dx\\
&\leq C\int_{\partial\mathbb{R}^{2}_{+}}|V(x,0)|^{2}dx= C\int_{\partial\mathbb{R}^{2}_{+}}|U(x,0)|^{2}dx.
\end{split}\end{equation}
The sharpness of the above inequality can be deduced from the sharpness of the sharp subcritical inequalities \eqref{addintq1}. In fact, we can pick the test function $u_k\in W^{\frac{1}{2},2}(\mathbb{R})$ satisfying $\|(-\Delta)^{\frac{1}{4}}u_k\|_{2}\leq 1$ such that  $$\lim\limits_{k\rightarrow \infty}\frac{\int_{\mathbb{R}}(\exp(\pi|u_k|^{2})-1)dx}{\|u_k\|_2^2}=\infty.$$
Define $U_k(x,y)$ as the harmonic extension of $u_k$, through Lemma \ref{addlem3}, we know that
$\int_{\mathbb{R}^2_{+}}|\nabla U_k(x,y)|^2dxdy=\int_{\mathbb{R}}|(-\Delta)^{\frac{1}{4}}u_k|^2dx\leq 1$, then it follows that
$$\lim\limits_{k\rightarrow\infty}\frac{\int_{\partial\mathbb{R}^2_{+}}(\exp(\pi|U_k|^{2})-1)dx}{ \int_{\partial \mathbb{R}_{+}^{2}}|U_{k}(x, 0)|^{2} dx}=\lim\limits_{k\rightarrow \infty}\frac{\int_{\mathbb{R}}(\exp(\pi|u_k|^{2}))-1)dx}{\|u_k\|_2^2}=\infty,$$
which implies  the sharpness of the inequality \eqref{addsub}.
\vskip 0.1cm

Next, we show the existence of extremals for subcritical trace Trudinger-Moser inequalities. According to Lemma \ref{addlem2}, we know for any $\beta<\pi$, there exists $u_0(x)\in W^{\frac{1}{2},2}(\mathbb{R})$ satisfying $\|(-\Delta)^{\frac{1}{4}}u_0\|_2^2=1$ such that
$$\frac{\int_{\mathbb{R}}[\exp(\beta|u_0|^2)-1]dx}{\int_{\mathbb{R}}|u_0|^2dx}=\sup_{u\in  W^{\frac{1}{2},2}(\mathbb{R}),\|(-\Delta)^{\frac{1}{4}}u\|_2^2=1}\frac{\int_{\mathbb{R}}[\exp(\beta|u|^2)-1]dx}{\int_{\mathbb{R}}|u|^2dx}.$$
Define $U(x,y)$ as the harmonic extension of $u_0(x)$, through Lemma \ref{addlem3}, we see that $U(x,y)$ is an extremal function for \eqref{addsub}.
\vskip0.1cm

Finally, we prove that the extremals of the inequality \eqref{addsub} must be a harmonic extension of
some radial function $u$ in $W^{\frac{1}{2},2}(\mathbb{R})$. In fact, assume that $U(x,y)$ is the extremal function of the inequality \eqref{addsub},
we define $W(x,y)$ as the harmonic extension of $U(x,0)$ and can easily check that $W(x,y)$ is also the extremal function of the inequality with the
$$\int_{\mathbb{R}^2_{+}}|\nabla U|^2dxdy=\int_{\mathbb{R}^2_{+}}|\nabla W|^2dxdy=1.$$
According to the Dirichlet principle, we know that $\int_{\mathbb{R}^2_{+}}|\nabla U|^2dxdy=\int_{\mathbb{R}^2_{+}}|\nabla W|^2dxdy$ if and only if
$U=W$. Furthermore, it is also easy to check that $U(x,0)$ is the extremal functions of the inequality \eqref{addintq1}. According to the Lemma \ref{addlem2}, the $U(x,0)$ must be radially symmetric with respect to some point $x_0 \in \mathbb{R}$. Hence the extremals of the \eqref{addsub} must be a harmonic extension of some radial function.
\medskip

\emph{Proof of Theorem \ref{addthm2}:}
For any $U\in W^{1,2}(\overline{\mathbb{R}^2_{+}})$ satisfying $$\int_{\mathbb{R}^2_{+}}|\nabla U(x,y)|^2dxdy+\int_{\partial\mathbb{R}^2_{+}}|U(x,0)|^2dx\leq 1,$$  we define $V(x,y)$ as the harmonic extension of $U(x,0)$.
By the Dirichlet principle and Lemma \ref{addlem2}, we derive that $\int_{\partial\mathbb{R}^2_{+}}(|(-\Delta)^{\frac{1}{4}}V(x,0)|^2+|V(x,0)|^2)dx\leq 1$. With the help of the critical fractional Trudinger-Moser inequality in $W^{\frac{1}{2},2}(\mathbb{R})$ (see \cite{Iula}), we conclude that
\begin{equation}\begin{split}
\int_{\partial\mathbb{R}^{2}_{+}}(\exp(\pi|U(x,0)|^{2})-1)dx&=\int_{\partial\mathbb{R}^{2}_{+}}(\exp(\pi|V(x,0)|^{2})-1)dx\leq C,\\
\end{split}\end{equation}
which accomplishes the proof of the inequality \eqref{addcri}. The existence of extremal of the inequality \eqref{addcri} relies on the existence of the critical fractional Trudinger-Moser inequality in $W^{\frac{1}{2},2}(\mathbb{R})$ which was recently established by Mancini and Martinazzi in \cite{MiMar}. The proof is similar to the proof of the existence of extremals of the subcritical trace Trudinger-Moser inequalities, we omit the details here.
\medskip

\section{Trace Adams inequalities on the half space $\mathbb{R}^{4}_{+}$}
In this section, we are devoted to establishing the sharp trace Adams inequalities and the existence of their extremal functions. The method is based on the bi-harmonic, Fourier rearrangement of \cite{Lenza} and sharp fractional Adams inequalities. We note that he method of Fourier rearrangement to establish the existence of extremals to Adams inequalities have  been used earlier in \cite{CLZ1} and \cite{CLZhu}.

For this purpose, we need the following lemma.
\begin{lemma}\label{lem1}(\cite{Fon})
For $0<\beta<6\pi^2$, then there exists a positive constant $C$ such that for all functions $u \in W^{\frac{3}{2},2}(\mathbb{R}^3)$ with $\|(-\Delta)^{\frac{3}{4}} u\|_{2}=1$, the following inequality holds.
\begin{equation}\label{intq1}
\int_{\mathbb{R}^{3}}(\exp(\beta|u|^{2})-1)dx
\leq C\int_{\mathbb{R}^3}|u|^{2}dx,
\end{equation}
 Moreover, the constant $6\pi^2$ is sharp in the sense that the inequality fails if the constant $\beta$ is replaced by any $\beta\geq6\pi^2$.
\end{lemma}
Next, we will prove the following
\begin{lemma} \label{lem2}
There exists an extremal for the above subcritical fractional Adams inequality \eqref{intq1}. Furthermore, all the extremals of inequality \eqref{intq1}
must be radially symmetric with respect to some point $x_0 \in \mathbb{R}^3$.
\end{lemma}
\begin{proof}
Define the sharp constant $\mu_{\beta}$ by
$$\mu_{\beta}:=\sup\limits_{u\in W^{\frac{3}{2},2}(\mathbb{R}^{3}),\|(-\Delta)^{\frac{3}{4}} u\|_{2}=1}F_{\beta}(u),$$
where
$$F_{\beta}(u):=\frac{\int_{\mathbb{R}^{3}}(\exp(\beta|u|^{2})-1)dx}{\|u\|_{2}^{2}}.$$
We first employ the Fourier rearrangement tools to prove that there exists a radially maximizing sequence for $\mu_{\beta}$.
In fact, assume that $(u_k)$ is a maximizing sequence for $\mu_{\beta}$, that is $$\|(-\Delta)^{\frac{3}{4}}u_k\|_{2}=1,\ \  \lim\limits_{k\rightarrow \infty}F_{\beta}(u_k)\rightarrow\mu_{\beta}.$$
Define $u_k^{\sharp}$ by $u_k^{\sharp}=\mathcal{F}^{-1}\{(\mathcal{F}(u_k))^{\ast}\}$, where $\mathcal{F}$ denotes
the Fourier transform on $\mathbb{R}^{3}$ (with its inverse $\mathcal{F}^{-1}$) and
$f^{\ast}$ stands for the Schwarz symmetrization of $f$. Using the property of
the Fourier rearrangement from \cite{Lenza}, one can derive that
\[
\Vert(-\Delta)^{\frac{3}{4}} u_k^{\sharp}\Vert_{2}\leq\Vert(-\Delta)^{\frac{3}{4}} u_k\Vert_{2},\ \Vert u_k^{\sharp
}\Vert_{2}=\Vert u_k\Vert_{2},\ \Vert u_k^{\sharp
}\Vert_{q}\geq \|u_k\|_{q}\ (q>2).
\]
Hence, $\lim\limits_{k\rightarrow \infty}F_{\beta}(u_k)\leq \lim\limits_{k\rightarrow \infty}F_{\beta}(u_k^{\sharp})$, which implies that
$(u^{\sharp}_k)$ is also the maximizing sequence for $\mu_{\beta}$.
Constructing  a new function sequence $(v_k)$ defined by $v_k(x):=u_k(\|u_k\|_{2}^{\frac{2}{3}}x)$ for $x\in \mathbb{R}^{3}$, one can easily verify that
$(v_k)$ is also a maximizing sequence for $\mu_{\beta}$ with $\|(-\Delta)^{\frac{3}{4}} v_k\|_{2}=1$ and $\|v_k\|_{2}=1$.
Note $(v_k)$ is bounded in $W^{\frac{3}{2},2}(\mathbb{R}^{3})$,
up to a sequence, we may assume that
$$ v_k\rightharpoonup v {\rm\ in}\  W^{\frac{3}{2},2}(\mathbb{R}^{3}),$$
thus $v$ satisfies that $\|v\|_2\leq 1$ and $\|(-\Delta)^{\frac{3}{4}}  v\|_2^2\leq1$.
Since $W^{\frac{3}{2},2}_{r}(\mathbb{R}^{3})$ (the collection of all radial functions in $W^{\frac{3}{2},2}(\mathbb{R}^{3})$) can be compactly imbedded into $L^{p}(\mathbb{R}^{3})$ for any $p>2$, we can claim that
 \begin{equation}\begin{split}\label{g.51}
&\lim_{k\rightarrow \infty}\int_{\mathbb{R}^{3}}(\exp(\beta|v_k|^{2})-1-\beta
|v_k|^2)dx=\int_{\mathbb{R}^{3}}(\exp(\beta|v|^{2})-1-\beta
|v|^2)dx.\\
\end{split}\end{equation}
For simplicity, we define $\Psi(\tau):=\exp(\tau)-1-\tau$ and  $\Phi(\tau):=\exp(\tau)-1$ for $\tau>0$,  then we can rewrite \eqref{g.51} as
\begin{equation}\label{a.2}
\lim_{k\rightarrow \infty}\int_{\mathbb{R}^{3}}\Psi(\beta|u_k|^{2})dx=\int_{\mathbb{R}^{3}}\Psi(\beta|u|^{2})dx.
\end{equation}
Hence, it follows from the mean value theorem and the convexity of the function $\Psi$ that
\begin{equation}\begin{split}\label{a.4}
&|\Psi(\beta|u_k|^{2})-\Psi(\beta|u|^{2})|\\
&\ \ \lesssim \Phi(\theta\beta|u_k|^{2}+(1-\theta)\beta|u|^{2})(|u|+|u_k|)|u_k-u|\\
&\ \ \lesssim (|u_k|+|u|)(\Phi(\beta|u_k|^{2})+\Phi(\beta|u|^{2}))|u_k-u|,\\
\end{split}\end{equation}
where $\theta\in [0,1]$.
This together with the fractional Adams inequality  leads to
\begin{equation}\begin{split}\label{a.5}
&|\int_{\mathbb{R}^{3}}(\Psi(\beta|u_k|^{2})-\Psi(\beta|u|^{2}))dx|\\
&\ \ \lesssim\int_{\mathbb{R}^{3}}(|u_k|+|u|)(\Phi(\beta|u_k|^{2})+\Phi(\beta|u|^{2}))|u_k-u|dx\\
&\ \ \lesssim\||u_k|+|u|\|_{L^{a}(\mathbb{R}^{3})}\|\Phi(\beta|u_k|^{2})+\Phi(\beta|u|^{2})\|_{L^b(\mathbb{R}^{n})}\|u_k-u\|_{L^c(\mathbb{R}^{n})}\\
&\ \ \lesssim\|u_k-u\|_{L^c(\mathbb{R}^{3})},
\end{split}\end{equation}
where the constants $b>1$ sufficiently close to $1$ and $\frac{1}{a}+\frac{1}{b}+\frac{1}{c}=1$.
Note that $W^{\frac{3}{2},2}(\mathbb{R}^{3})$ can be compactly imbedded into $L^{r}(\mathbb{R}^{3})$ for any $r>2$, we derive that $$\lim_{k\rightarrow \infty}\int_{\mathbb{R}^{3}}\Psi(\beta|u_k|^{2})dx=\int_{\mathbb{R}^{3}}\Psi(\beta|u|^{2})dx,$$
which accomplishes the proof of \eqref{g.51}.
\vskip0.1cm

Then it follows form \eqref{g.51} that
\begin{equation}\begin{split}\label{26.3.1}
\mu_{\beta}&=F_{\beta}(v_k)+o(1)\\
&=\int_{\mathbb{R}^{3}}(\exp(\beta|v_k|^{2})-1)dx+o(1)\\
&=\beta+\int_{\mathbb{R}^{3}}(\exp(\beta|v_k|^{2})-1-\beta |v_k|^2)dx+o(1)\\
&=\beta+\int_{\mathbb{R}^{3}}(\exp(\beta|v|^{2})-1-\beta |v|^2)dx+o(1).\\
\end{split}\end{equation}
Next, we show $v\neq0$. Indeed, one can pick $u_0$ in $W^{\frac{3}{2},2}(\mathbb{R}^{3})$ satisfying $\|(-\Delta)^{\frac{3}{4}} u_0\|_2=1$ arbitrarily. Then, we have
\begin{equation*}\begin{split}
\mu_{\beta}\geq F_{\beta}(u_0)&=\frac{\int_{\mathbb{R}^{3}}(\exp(\beta|u_0|^2)-1)dx}{\|u_0\|_2^2}\\
&=\frac{\sum_{j=1}^{\infty}\frac{\beta^j}{j!}\|u_0\|_{2j}^{2j}}{\|u_0\|_2^2}\\
&=\beta+\frac{\sum_{j=2}^{\infty}\frac{\beta^j}{j!}\|u_0\|_{2j}^{2j}}{\|u_0\|_2^2}>\beta.\\
\end{split}\end{equation*}
Hence,
\begin{equation*}\begin{split}
\mu_{\beta}&\leq \beta +\frac{\int_{\mathbb{R}^{3}}(\exp(\beta|v|^2)-1-\beta|v|^2)dx}{\|v\|_2^2}\\
&=\frac{\int_{\mathbb{R}^{3}}(\exp(\beta|v|^2)-1)dx}{\|v\|_2^2}=F_{\beta}(v).
\end{split}\end{equation*}
Therefore, it remains to show $\|(-\Delta)^{\frac{3}{4}} v\|_2^2=1$. Recall that $\|(-\Delta)^{\frac{3}{4}} v\|_2^2\leq1$, it suffices to show that $\|(-\Delta)^{\frac{3}{4}} v\|_2^2\geq1$. Through the definition of $\mu_{\beta}$, one can obtain that
\begin{equation}\label{26.2.3}\begin{split}
\mu_{\beta}&\geq F_{\beta}(\frac{v}{\|(-\Delta)^{\frac{3}{4}} v\|_2})\\
&=\sum_{j=1}^{\infty}\frac{\beta^j}{j!}\frac{\| v\|_{2j}^{2j}}{\| v\|_2^2}\|(-\Delta)^{\frac{3}{4}} v\|_2^{2-2j}\\
&\geq\beta+\frac{\beta^2}{2}\frac{\| v\|_{4}^{4}}{\| v\|_2^2}\|(-\Delta)^{\frac{3}{4}} v\|_2^{-2}+
\sum_{j=2}^{\infty}\frac{\beta^j}{j!}\frac{\| v\|_{2j}^{2j}}{\| v\|_2^2}\\
&=F_{\beta}(v)+\frac{\beta^2}{2}\frac{\| v\|_{4}^{4}}{\| v\|_2^2}(\|(-\Delta)^{\frac{3}{4}} v\|_2^{-2}-1)\\
&\geq \mu_{\beta}+\frac{\beta^2}{2}\frac{\| v\|_{4}^{4}}{\| v\|_2^2}(\|(-\Delta)^{\frac{3}{4}} v\|_2^{-2}-1)\\
\end{split}\end{equation}
which implies that $\|(-\Delta)^{\frac{3}{4}} v\|_2^2\geq 1$. Thus, $v$ is a maximizer. Next, we prove that all the extremals of the inequality \eqref{intq1}
must be radially symmetric with respect to some point $x_0 \in \mathbb{R}^3$. Assume that $u$ is a maximizer
for the inequality \eqref{intq1}, we easily see that $u^{\sharp}$ is also a maximizer for the inequality \eqref{intq1} with
$\Vert(-\Delta)^{\frac{3}{4}} u^{\sharp}\Vert_{2}=\Vert(-\Delta)^{\frac{3}{4}} u\Vert_{2}$ and $\|u^{\sharp}\|_{L^q}=\|u\|_{L^q}$ for even $q$. Using the property
of the Fourier rearrangement from \cite{Lenza}, we conclude that
\[
u(x)=e^{i\alpha}u^{\sharp}(x-x_{0})\ \ \mbox{for any}\ x\in\mathbb{R}^{3}%
\]
with some constants $\alpha\in\mathbb{R}$ and $x_{0}\in\mathbb{R}^{3}$. That
is to say that $u$ is radially symmetric and real valued up to translation and
constant phase. Then we accomplish the proof of Lemma \ref{lem2}.
\end{proof}

\begin{lemma}\label{lem3}(\cite{NNP})
Given a function $u\in W^{\frac{3}{2},2}(\mathbb{R}^3)$.  If we assume that function $U\in W^{2,2}(\overline{\mathbb{R}^{4}_{+}})$ satisfying
$$(-\Delta)^2U(x,y)=0$$
on the upper half space $\mathbb{R}^{4}_{+}$ and the boundary condition
$$U(x,0)=u(x),\ \ \partial_{y}U(x,y)|_{y=0}=0,$$ then we have the following identity
\begin{equation}\label{identity1}
\int_{\mathbb{R}^4_{+}}|\Delta U(x,y)|^2dxdy=2\int_{\mathbb{R}^3}|(-\Delta)^{\frac{3}{4}}u|^2dx.
\end{equation}
\end{lemma}

Now, we are in the position to prove Theorem \ref{thm1}.
\vskip0.1cm

\emph{Proof of Theorem \ref{thm1}:}
For any $U\in W^{2,2}(\overline{\mathbb{R}^4_{+}})$ satisfying the Neumann boundary condition $\partial_{y}U(x,y)|_{y=0}=0$, we define $V(x,y)$ as the bi-harmonic extension of $U(x,0)$, that is to say $V(x,y)\in W^{2,2}(\overline{\mathbb{R}^4_{+}})$ satisfies that
$$(-\Delta)^2V(x,y)=0$$
on the upper half space $\mathbb{R}^{4}_{+}$ and the boundary condition
$$V(x,y)|_{y=0}=U(x,y)|_{y=0},\ \ \partial_{y}V(x,y)|_{y=0}=0.$$
In fact, we can write $$V(x,y)=c\int_{\partial\mathbb{R}^{4}_{+}} \frac{y^3}{(|x-\widetilde{x}^2|+y^2)^{3}}U(\widetilde{x},0)d\widetilde{x}$$ by the Green formula. Furthermore, by the Dirichlet principle for bi-harmonic function, we derive that
\begin{equation}\label{energy1}
\int_{\mathbb{R}^4_{+}}|\Delta V(x,y)|^2dxdy\leq \int_{\mathbb{R}^4_{+}}|\Delta U(x,y)|^2dxdy.
\end{equation}
On the other hand, through Lemma \ref{lem3}, we get
\begin{equation}\label{identity-1}
\int_{\mathbb{R}^4_{+}}|\Delta V(x,y)|^2dxdy=2\int_{\partial\mathbb{R}^4_{+}}|(-\Delta)^{\frac{3}{4}}V(x,0)|^2dx.
\end{equation}
Combining \eqref{energy1} and \eqref{identity-1} and Lemma \ref{lem1}, we conclude that
\begin{equation}\begin{split}
\int_{\partial\mathbb{R}^{4}_{+}}(\exp(\beta|U(x,0)|^{2})-1)dx&=\int_{\partial\mathbb{R}^{4}_{+}}(\exp(\beta|V(x,0)|^{2})-1)dx\\
&\leq C\int_{\partial\mathbb{R}^{4}_{+}}|V(x,0)|^{2}dx= C\int_{\partial\mathbb{R}^{4}_{+}}|U(x,0)|^{2}dx.
\end{split}\end{equation}
The sharpness of the inequality \eqref{sub} can be similarly deduced from the sharpness of the sharp subcritical inequalities \eqref{intq1} as done in the proof of sharpness of inequality (\ref{addsub}).
\vskip0.1cm

Now, we prove the existence of extremals for the subcritical trace Adams inequalities.
From Lemma \ref{lem2}, we see that there exists $u_0(x)\in W^{\frac{3}{2},2}(\mathbb{R}^3)$ such that
$$C=\sup_{u\in  W^{\frac{3}{2},2}(\mathbb{R}^3),\|(-\Delta)^{\frac{3}{4}}\|_2^2=1}\frac{\int_{\mathbb{R}^{3}}(\exp(\beta|u_0|^2)-1)dx}{\int_{\mathbb{R}^{3}}|u_0|^2dx}$$ for any $\beta<6\pi^2$. Define $U(x,y)$ as the bi-harmonic extension of $u_0(x)$, it is easy to check that $U(x,y)$ is an extremal function for \eqref{sub} through Lemma \ref{lem3}.
\vskip0.1cm

Finally, we show that the extremals of the inequality \eqref{sub} must be a bi-harmonic extension of
some radial function $u$ in $W^{\frac{3}{2},2}(\mathbb{R}^3)$. In fact, assume that $U(x,y)$ is the extremal function of the inequality (\ref{sub}),
It is easy to check that the bi-harmonic extension $V(x,y)$ is also the extremal function of the inequality with the
$$\int_{\mathbb{R}^4_{+}}|\Delta U|^2dxdy=\int_{\mathbb{R}^4_{+}}|\Delta V|^2dxdy=1,$$
which implies that $\int_{\mathbb{R}^4_{+}}|\Delta U|^2dxdy=\int_{\mathbb{R}^4_{+}}|\Delta V|^2dxdy$ if and only if
$U=V$ through Dirichlet principle. On the other hand, it is also easy to see that $U(x,0)$ is the extremal functions of the inequality (\ref{intq1}). Applying Lemma \ref{lem2} again, we conclude that $U(x,0)$ must be radially symmetric with respect to some point $x_0 \in \mathbb{R}^3$.  This proves that the extremals of inequality \eqref{sub} must be a bi-harmonic extension of some radial function.
\medskip

\emph{Proof of Theorem \ref{thm2} and Theorem \ref{adam-exact}:}
For any $U\in W^{2,2}(\overline{\mathbb{R}^4_{+}})$ satisfying Neumann boundary condition $\partial_{y}U(x,y)|_{y=0}=0$ with $$\int_{\mathbb{R}^4_{+}}|\Delta U(x,y)|^2dxdy+\int_{\partial\mathbb{R}^4_{+}}|U(x,0)|^2dx\leq 1,$$ we derive that
\begin{equation}\label{energy2}
\int_{\mathbb{R}^4_{+}}|\Delta V(x,y)|^2dxdy\leq \int_{\mathbb{R}^4_{+}}|\Delta U(x,y)|^2dxdy,
\end{equation}
where $V(x,y)$ is the bi-harmonic extension of $U(x,0)$. This together with Lemma \ref{lem3} and \eqref{energy2} yields that $$\int_{\partial\mathbb{R}^4_{+}}(2|(-\Delta)^{\frac{3}{4}}V(x,0)|^2+|V(x,0)|^2)dx\leq 1.$$
Combining the critical fractional Adams inequalities in $W^{\frac{3}{2},2}\left(\mathbb{R}^3\right)$, we derive that
\begin{equation}\begin{split}
\int_{\partial\mathbb{R}^{4}_{+}}(\exp(12\pi^2|U(x,0)|^{2})-1)dx&=\int_{\partial\mathbb{R}^{4}_{+}}(\exp(12\pi^2|V(x,0)|^{2})-1)dx\leq C.
\end{split}\end{equation}
Similarly, one can deduce from the fractional Adams inequalities with the exact growth in $W^{\frac{2}{3}}(\mathbb{R}^3)$ (see Theorem \ref{fratru}) that
 \begin{equation}\begin{split}
\int_{\partial\mathbb{R}^{4}_{+}}\frac{(\exp(12\pi^2|U(x,0)|^{2})-1)}{1+|U(x,0)|^2}dx&=\int_{\partial\mathbb{R}^{4}_{+}}(\exp(12\pi^2|V(x,0)|^{2})-1)dx\leq C\int_{\partial\mathbb{R}^{4}_{+}}|U(x,0)|^{2}dx,
\end{split}\end{equation}
which gives the inequality \eqref{addadmexa}.

\section{Existence of ground states to the bi-Laplacian  with nonlinear
Neumann boundary conditions on the half space $\mathbb{R}^{4}_{+}$}
In this section, we consider the critical point of the functional $I_{\lambda}(U)$ on $E$ (see \eqref{functional}). For this purpose, we will first study the critical point of  the following functional $$J_{\lambda}(u)=\frac{1}{2}\big(2\int_{\mathbb{R}^3}|(-\Delta )^{\frac{3}{4}}u|^2dx+\int_{\mathbb{R}^3}|u|^2dx\big)
-\frac{\lambda}{24\pi^2}\int_{\mathbb{R}^3}\big(\exp(12\pi^2|u|^2)-1\big)dx,$$
and we show that the functional $J_{\lambda}(u)$ has a least energy critical point provided $0<\lambda<1$, that is,

\begin{proposition}\label{technical lemma}
 $m_{\lambda}=\inf\{{J_{\lambda}(u):J^{\prime}_{\lambda}(u)=0, u\in W^{\frac{3}{2},2}(\mathbb{R}^3)}\}$
can be achieved for  any $0<\lambda<1$.
\end{proposition}
\emph{Proof of Proposition \ref{technical lemma}}:
Motivated by the Pohazoev identity for the functional $J_{\lambda}$, we introduce the functional
$$G_{\lambda}(u)=\|u\|_2^2-\frac{\lambda}{12\pi^2}\int_{\mathbb{R}^3} \big(\exp(12\pi^2|u|^2)-1)dx=(1-\lambda)\|u\|_2^2-
\int_{\mathbb{R}^3}g_{\lambda}(u)dx,$$
where $g_{\lambda}(t)$ is defined as $g_{\lambda}(t)=\frac{\lambda}{12\pi^2} \big(\exp(12\pi^2|t|^2)-1-12\pi^2t^2)$
and the constrained minimization problem
\begin{equation}\begin{split}
A_{\lambda} &=\inf \Big\{\|(-\Delta)^{\frac{3}{4}}u\|_2^2 \ | \ u\in W^{\frac{3}{2},2}(\mathbb{R}^3),\ G_{\lambda}(u)=0 \Big\}\\
&=\inf \Big\{J_{\lambda}(u)\ | \ u\in W^{\frac{3}{2},2}(\mathbb{R}^3),\ G_{\lambda}(u)=0 \Big\}\leq m_{\lambda}.\\
\end{split}\end{equation}
Set $M=\{u\in W^{\frac{3}{2},2}(\mathbb{R}^3),\ G_{\lambda}(u)=0\}$, we point out that $M$ is not empty. In fact, let
$u_0\in W^{\frac{3}{2},2}(\mathbb{R}^3)$ be compactly supported and define
$$h(s):=G_{\lambda}(su_0)=s^2(1-\lambda)\|u_0\|_2^2-\int_{\mathbb{R}^3} g_{\lambda}(su_0)dx, \forall s >0.$$
From $\lim\limits_{t\rightarrow 0}\frac{g_\lambda(t)}{t^2}=0$ and $\lim\limits_{t\rightarrow +\infty}\frac{g_\lambda(t)}{t^2}=\infty$,  we have
$h(s)>0$ for $s>0$ small enough and $h(s)<0$ for $s>0$ sufficiently large.  Therefore, there exists $s_0>0$ satisfying
$h(s_0)=0$, which implies $s_0u_0 \in M$.

Next, we show that $A_{\lambda}$ can be attained by some function $u\in W^{\frac{3}{2},2}(\mathbb{R}^3)\backslash \{0\}$. For this, we need the following lemma:
\vskip0.1cm
\begin{lemma}\label{rem1}
There exists a radially minimizing sequence $\{u_k\}_k$ satisfying $\|u_k\|_2^2=1$ for $A_{\lambda}$.
\end{lemma}
\begin{proof}
Assume that $\{u_k\}_k$ is a minimizing sequence for $A_{\lambda}$, that is $u_k \in M$ satisfying
$$\lim\limits_{k\rightarrow \infty}\|(-\Delta)^{\frac{3}{4}} u_k\|_2^2=A_{\lambda}.$$
Define $u_k^{\sharp}$ by the Fourier rearrangement of $u_k$. Using the property of the Fourier rearrangement from \cite{Lenza}, one can derive that
$$\|(-\Delta)^{\frac{3}{4}} u_k^{\sharp}\|_2\leq \|(-\Delta)^{\frac{3}{4}} u_k\|_2,\ \, \|u_k^{\sharp}\|_2^2=\|u_k\|_2^2,$$
and $$\int_{\mathbb{R}^{3}}\big(\exp(12\pi^2{(u_k^{\sharp})}^2)-1\big)dx\geq\int_{\mathbb{R}^3}\big(\exp(12\pi^2u_k^2)-1\big)dx.$$
Then it follows that $$(1-\lambda)\|u_k^{\sharp}\|_2^2=(1-\lambda)\|u_k\|^2_2=\int_{\mathbb{R}^3}g_{\lambda}(u_k)dx\leq \int_{\mathbb{R}^3}g_{\lambda}(u_k^{\sharp})dx.$$
Hence if we set
$$\gamma(t)=(1-\lambda)\|t{u_k^{\sharp}}\|_2^2-\int_{\mathbb{R}^3}g_{\lambda}(t{u_k^{\sharp}})dx,$$
then we have $\gamma(1)\leq 0$ while $\gamma(t)>0$ for $t>0$ sufficiently small. Therefore, there exists $t_k \in (0,1]$ such that
$\gamma(t_k)=0$, that is $t_k{u_k^{\sharp}} \in M$. We obtain
\begin{equation}
A_\lambda\leq I(t_ku_k^{\sharp})=\|(-\Delta)^{\frac{3}{4}}(t_ku_k^{\sharp})\|_2^2\leq t_k^2\|(-\Delta)^{\frac{3}{4}}u_k\|_2^2\leq I(u_k)=A_\lambda+o(1).
\end{equation}
This implies that $\{v_k\}=\{t_k{u_k^{\sharp}}\}_k$ is a radial minimizing sequence for $A_{\lambda}$.  Let $\tilde{v}_k={u_k^{\sharp}}(\|v_k\|_2^{\frac{1}{2}}x)$,
direct computations yield that $\|\tilde{v}_k\|_2=1$, $\tilde{v}_k \in M$ and $\|(-\Delta)^{\frac{3}{4}}v_k\|_2=\|(-\Delta)^{\frac{3}{4}}\tilde{v}_k\|_2$.
\end{proof}

In order to study the attainability of $A_\lambda$, we also need the following compactness result.

\begin{lemma}[Compactness]\label{comp}
Suppose that $g: \mathbb{R}\rightarrow [0,+\infty)$ is a continuous function and define
$$G(u)=\int_{\mathbb{R}^3}g(u)dx.$$
Then for any $K>0$, if
$$\ \lim\limits_{t\rightarrow +\infty}\exp(-\frac{1}{K}|t|^{2})g(t)<+\infty,\ \
\lim\limits_{t\rightarrow 0}|t|^{-2}g(t)=0,$$
then for any bounded sequence $\{u_k\}_k \in W^{\frac{3}{2},2}_{rad}(\mathbb{R}^3)$ satisfying $\lim\limits_{k\rightarrow +\infty}\int_{\mathbb{R}^3}|(-\Delta)^{\frac{3}{4}}u_k|^2 dx< 6\pi^2 K$
and weakly converging to some $u$, we have that $G(u_k)\rightarrow G(u)$.
\end{lemma}
\begin{proof}
Note that $\{u_k\}_k$ is a radial sequence in $W^{\frac{3}{2},2}(\mathbb{R}^3)$, then by radial Lemma, there exists $\delta>0$ such that
\begin{equation}\begin{split}
|u_k(r)|\lesssim \frac{1}{r^{\delta}}.
\end{split}\end{equation}
Hence $u_k(r)\rightarrow 0$ as $r\rightarrow \infty$ uniformly with respect to $k$. This together with the subcritical Adams inequality \eqref{intq1} yields that for any $\varepsilon>0$, there exists $R>0$ such that
\begin{equation}\label{s3}
\int_{\mathbb{R}^3\setminus B_{R}}g(u_k)dx\lesssim \varepsilon \int_{\mathbb{R}^3}|u_k|^2dx\lesssim \varepsilon,
\ \ \int_{\mathbb{R}^3\setminus B_{R}}g(u)dx\lesssim \varepsilon.
\end{equation}
On the other hand, since $\lim\limits_{k\rightarrow +\infty}\int_{\mathbb{R}^3}|(-\Delta)^{\frac{3}{4}}u_k|^2 dx< 6\pi^2K$, then there exists some $\gamma>0$ such that $\lim\limits_{k\rightarrow +\infty}\int_{\mathbb{R}^3}|(-\Delta)^{\frac{3}{4}}u_k|^2 dx< \gamma<6\pi^2K$.

Through $\lim\limits_{t\rightarrow +\infty}\exp(-\frac{1}{K}|t|^{2})g(t)<+\infty$, we derive that for $\varepsilon>0$, there exists $L$ independent of $k$ such that
$$\int_{|u_k|>L}g(u_k)dx\lesssim \varepsilon \int_{|u_k|>L}\exp(\frac{\gamma}{K}|\frac{u_k}{\|(-\Delta)^{\frac{3}{4}}u_k\|_2}|^2)dx$$
and
$$\int_{|u_k|>L}g(u)dx\lesssim \varepsilon \int_{|u_k|>L}\exp(\frac{\gamma}{K}|\frac{u}{\|(-\Delta)^{\frac{3}{4}}u\|_2}|^2)dx,$$
Applying the sharp subcritical Adams inequality \eqref{intq1}, we derive that
\begin{equation}\label{s4}
\int_{|u_k|>L}g(u_k)dx\lesssim \varepsilon\int_{\mathbb{R}^3}|u_k|^2dx\lesssim \varepsilon,\ \
\int_{|u_k|>L}g(u)dx\lesssim \varepsilon.
\end{equation}
Combining \eqref{s3} and \eqref{s4}, we conclude that
\begin{equation}\begin{split}
\lim_{k\rightarrow \infty}|G(u_k)-G(u)|&\leq \int_{\mathbb{R}^3\setminus B_{R}}(g(u_k)-g(u))dx+\int_{B_{R}}(g(u_k)-g(u))dx\\
& \lesssim \varepsilon +\lim_{k\rightarrow \infty} \Big|\int_{B_{R}\cap\{|u_k|>L\}}g(u_k)dx-\int_{B_{R}\cap\{|u_k|>L\}}g(u)dx\Big|\\
&\ \ + \lim_{k\rightarrow \infty}\Big|\int_{B_{R}\cap\{|u_k|\leq L\}}g(u_k)dx-\int_{B_{R}\cap\{|u_k|\leq L\}}g(u)dx\Big|\\
&\lesssim \varepsilon+\lim_{k\rightarrow \infty}\Big|\int_{B_{R}\cap\{|u_k|\leq L\}}g(u_k)dx-\int_{B_{R}\cap\{|u_k|\leq L\}}g(u)dx\Big|\\
&\lesssim \varepsilon,
\end{split}\end{equation}
which gives the proof by the Lebesgue dominated convergence theorem.
\end{proof}

With this compactness result, we can prove the following

\begin{lemma}\label{rem2}
If $A_{\lambda}< \frac{1}{2}$, then $A_{\lambda}$ can be attained and $A_{\lambda}=I_{\lambda}(u)$.
\end{lemma}
\begin{proof}
Let $\{u_k\}_k$ is a radial minimizing sequence for $A_{\lambda}$, that is $u_k \in M$ satisfying
$$\lim\limits_{k\rightarrow \infty}\|(-\Delta)^{\frac{3}{4}} u_k\|_2^2=A_{\lambda} \ \ \text{and}\ \ \|u_k\|_2^2=1.$$ We also assume that $u_k\rightharpoonup u$ in $W^{\frac{3}{2},2}(\mathbb{R}^3)$.
We first prove that $A_{\lambda}>0$. By way of contradiction, we assume that $A_{\lambda}=0$. That is $\lim\limits_{k\rightarrow \infty}\|(-\Delta)^{\frac{3}{4}} u_k\|_2^2=0$, which implies that $u=0$.
Since $$\ \lim\limits_{t\rightarrow +\infty}\exp(-c|t|^{2})g_{\lambda}(t)=0\ \text{for\ any }c>12\pi^2,\ \
\lim\limits_{t\rightarrow 0}|t|^{-2}g_{\lambda}(t)=0.$$ It follows from the Lemma \ref{comp} that
$$\int_{\mathbb{R}^3}g_{\lambda}(u_k)dx=\int_{\mathbb{R}^3}g_{\lambda}(u)dx.$$ On the other hand,  it follows from $u_k \in M$ and $\|u_k\|_2^2=1$ that
$$0<(1-\lambda)\leq \lim\limits_{k\rightarrow \infty}(1-\lambda)\|u_k\|_2^2=\lim\limits_{k\rightarrow \infty}\int_{\mathbb{R}^3}g_{\lambda}(u_k)dx=\int_{\mathbb{R}^3}g_{\lambda}(u)dx,$$
which contradicts $u=0$. This proves that $A_{\lambda}>0$.
\medskip

Now are in position to prove that if $A_{\lambda}< \frac{1}{2}$, then $A_{\lambda}$ could be attained.
Under the assumption of Lemma \ref{lem2}, we have $\lim_{k\rightarrow \infty}\|\Delta u_k\|_2^2=A_{\lambda}<\frac{1}{2}$. Set $K=\frac{A_{\lambda}}{6\pi^2}$,
obviously, $$\ \lim\limits_{t\rightarrow +\infty}\exp(-\frac{1}{K}|t|^{2})g_{\lambda}(t)=0, \ \
\lim\limits_{t\rightarrow 0}|t|^{-2}g_{\lambda}(t)=0.$$ It follows from Lemma \ref{comp} that $$\int_{\mathbb{R}^3}g_{\lambda}(u_k)dx=\int_{\mathbb{R}^3}g_{\lambda}(u)dx.$$
Hence $$(1-\lambda)=\lim\limits_{k\rightarrow \infty}(1-\lambda)\|u_k\|_2^2=\lim\limits_{k\rightarrow \infty}\int_{\mathbb{R}^3}g_{\lambda}(u_k)dx=\int_{\mathbb{R}^3}g_{\lambda}(u)dx$$ and $$\|(-\Delta)^{\frac{3}{4}} u\|_2^2\leq \lim\limits_{k\rightarrow \infty}\|(-\Delta)^{\frac{3}{4}}u_k\|_2^2=A_{\lambda}.$$ In order to show $u$ is minimizer for $A_{\lambda}$, it suffices to show that $G_{\lambda}(u)=0$.
Since \begin{equation}
G_{\lambda}(u)=(1-\lambda)\|u\|_2^2-\int_{\mathbb{R}^3}g_{\lambda}(u)dx\leq \lim\limits_{k\rightarrow \infty}(1-\lambda)\|u_k\|_2^2-\int_{\mathbb{R}^3}g_{\lambda}(u_k)dx=\lim\limits_{k\rightarrow \infty}G_{\lambda}(u_k)=0.
\end{equation}
If we define
$$h(t)=G_{\lambda}(tu)=(1-\lambda)\|tu\|_2^2-\int_{\mathbb{R}^3}g(tu)dx,$$ then $h(1)\leq 0$ and from $\lim_{k\rightarrow \infty}g(t)=o(t^2)$ , we deduce that
$h(t)<0$ for small $t>0$. Consequently, there exists $s_0 \in (0,1]$ such that $G_{\lambda}(s_0u)=0$. Then we have
$$A_{\lambda}\leq \|(-\Delta)^{\frac{3}{4}} s_0u\|_2^2=s_0^2\|(-\Delta)^{\frac{3}{4}} u\|_2^2\leq s_0^{2}A_{\lambda},$$
which proves that $s_0=1$ and  $\|(-\Delta)^{\frac{3}{4}} u\|_2^2=A_{\lambda}$. Then we accomplish the proof.
\end{proof}
In order to obtain the attainability of $A_{\lambda}$, we introduce
the fractional Adams ratio
$$C_{A,\lambda}^{L}=\sup\{\frac{1}{\|u\|_2^2}\int_{\mathbb{R}^3}g_{\lambda}(u)dx|\ u\in W^{\frac{3}{2},2}(\mathbb{R}^3), \|(-\Delta)^{\frac{3}{4}} u\|_2^2\leq L\}.$$
The fractional Adams threshold $R(g_{\lambda})$ is defined as
$$R(g_{\lambda})=\sup\{L>0\ |\  C_{A,\lambda}^{L} <+\infty\}.$$
Obviously, according subcritical fractional Adams inequality \eqref{intq1}, we know $R(g_{\lambda})=\frac{1}{2}$
and $C_{A,\lambda}^{*}:= C_{A,\lambda}^{R(g_{\lambda})}=\infty$. By using this fact, we can claim that \begin{equation}\label{small}
A_{\lambda}<\frac{1}{2}.
\end{equation}
Indeed, note that $C_{A,\lambda}^{*}=\infty$, hence there exists $u_0\in W^{\frac{3}{2},2}(\mathbb{R}^3)$ such that
$$(1-\lambda)< \frac{1}{\|u_0\|_2^2}\int_{\mathbb{R}^3}g_{\lambda}(u_0)dx,\ \ \|(-\Delta)^\frac{3}{4} u_0\|_2^2\leq \frac{1}{2}.$$
Hence $G_{\lambda}(u_0)=(1-\lambda)\|u_0\|_2^2-\int_{\mathbb{R}^3}g_{\lambda}(u_0)dx<0$. Then there exists $s_0\in (0,1)$ such that
$s_0u_0\in M$, which yields that
$$A_{\lambda}\leq \|(-\Delta)^{\frac{3}{4}}(s_0u_0)\|_2^2=s_0^2\|(-\Delta)^{\frac{3}{4}} u_0\|_2^2\leq \frac{1}{2}s_0^2<\frac{1}{2}.$$
Then the claim is finished. Now, we come to
\vskip0.1cm

\emph{Completion of  the proof of Proposition \ref{technical lemma}}:
Combining Lemmas \ref{rem1}, \ref{rem2} and \eqref{small}, we see that $A_{\lambda}$ can be attained. Under  a suitable change
of scale, it is easy to see that $m_{\lambda}$ can also be attained.
\medskip

In view of  Proposition \ref{technical lemma}, we can give the proof of Theorem \ref{thm3}.
\vskip0.1cm

\emph{Proof of Theorem \ref{thm3}}: Let $v(x)\in W^{\frac{3}{2},2}(\mathbb{R}^3)$ be the minimum point of the functional $J_{\lambda}$ on the submanifold $F:=\{J^\prime_{\lambda}(u)=0\ |u\in W^{\frac{3}{2},2}(\mathbb{R}^3)\}$ and  define $V(x,y)$ as the bi-harmonic extension of $v(x)$, that is to say $V(x,y)$ satisfies that
$$(-\Delta)^2V(x,y)=0$$
on the upper half space $\mathbb{R}^{4}_{+}$ and the boundary condition.
$$V(x,0)=v(x),\ \ \partial_{y}V(x,y)|_{y=0}=0.$$
We will show that $V(x,y)$ is the ground state of equation \eqref{Pro}, that is to prove that $V(x,y)$
is the minimum point of the functional $I_{\lambda}$ on the submanifold $\tilde{F}:= \{I^\prime_{\lambda}(U)=0\ |U\in E\}$.
 \medskip

 Assume that $U(x,y)$ is any solution of equation \eqref{Pro}, we only need to verify that $I_{\lambda}(V)\leq I _{\lambda}(U)$. Through equation \eqref{Pro}, we easily see that $U(x,y)$ is the harmonic extension of $u(x)=U(x,0)$ and $u(x)$ is the critical point of functional $J_{\lambda}$ in $W^{\frac{3}{2},2}(\mathbb{R}^3)$. In view of Lemma \ref{lem3}, we derive that
\begin{equation}\label{identity1}
\int_{\mathbb{R}^4_{+}}|\Delta V(x,y)|^2dxdy=2\int_{\mathbb{R}^3}|(-\Delta)^{\frac{3}{4}}v(x)|^2dx,\ \ \int_{\mathbb{R}^4_{+}}|\Delta U(x,y)|^2dxdy=2\int_{\mathbb{R}^3}|(-\Delta)^{\frac{3}{4}}u(x)|^2dx.
\end{equation}
Then it follows that $I_{\lambda}(V)=J_{\lambda}(v)$ and $I_{\lambda}(U)=J_{\lambda}(u)$.
This together with $v$ being the least energy critical point of the $J_{\lambda}$ on $F$ implies that $$I_{\lambda}(V)=J_{\lambda}(v)\leq J_{\lambda}(u)=I_{\lambda}(U),$$
that is to say $V$ is the ground-state solution for equation \eqref{Pro} with nonlinear
Neumann boundary conditions on the half space.

\section{Proofs of   Lemma \ref{lem1.1} and Theorems \ref{thm1.1} and \ref{thm2.1} }\label{S5}

The main purpose of this section is to establish Lemma \ref{lem1.1} which is the essential technical ingredient
in proving Theorems \ref{thm1.1} and \ref{thm2.1} for $m>2$ by adapting the same argument in establishing Theorems \ref{thm1} and \ref{thm2} in the case of $m=2$.

We now give the proof of Lemma \ref{lem1.1}.

\begin{proof}
It is known that the solution $U(x,y)$ is unique. By using the the generalized
Poisson kernel, it has been shown in \cite{Yang} that the explicit formula of $U$ is given by
\begin{equation*}
\begin{split}
U(x,y)=\pi^{-m+\frac{1}{2}}
\frac{\Gamma(2m-1)}{\Gamma(m-\frac{1}{2})}
\int_{\mathbb{R}^{2m-1}}\frac{y^{2m-1}}{\big(|x-\xi|^{2}+y^{2}\big)^{2m-1}}u(\xi)d\xi.
\end{split}
\end{equation*}
Furthermore, $U(x,y)$ satisfying the following Neumann boundary condition
\begin{equation}\label{a1}
\begin{split}
\Delta^{k}U(x,y)|_{y=0}=&\frac{\Gamma(m)\Gamma(m-\frac{1}{2}-k)}{\Gamma(m-1-k)\Gamma(m-\frac{1}{2})}\Delta^{k}_{x}u,
\;\;\;0\leq k\leq \left[\frac{m-1}{2}\right];\\
\partial_{y}\Delta^{k}U(x,y)|_{y=0}=&0,\;\;\;\;\;\;\;\;\;\;\;\;\;\;\;\;\;\;\;\;\;\;\;\;\;\;\;\;\;\;\;\;\;\;\;\;\;\;
\;\;\;\;\;\;\;\;0\leq k\leq\left[\frac{m-2}{2}\right];\\
\partial_{y}\Delta^{m-1}U(x,y)|_{y=0}=&(-1)^{m-1}\frac{\Gamma(m)\Gamma(\frac{1}{2})}{\Gamma(m-\frac{1}{2})}(-\Delta_{x})^{m-\frac{1}{2}}u.
\end{split}
\end{equation}
By   Green's formula and (\ref{a1}), we have
\begin{equation*}
\begin{split}
0=&\int_{\mathbb{R}_{+}^{2m}}U\Delta^{m}Udxdy\\
=&\int_{\mathbb{R}_{+}^{2m}}\Delta U\Delta^{m-1}Udxdy+\int_{\mathbb{R}^{2m-1}}U\partial_{y}\Delta^{m-1}Udx-
\int_{\mathbb{R}^{2m-1}}\partial_{y}U\Delta^{m-1}Udx\\
=&\int_{\mathbb{R}_{+}^{2m}}\Delta U\Delta^{m-1}Udxdy+(-1)^{m}\frac{\Gamma(m)\Gamma(\frac{1}{2})}{\Gamma(m-\frac{1}{2})}
\int_{\mathbb{R}^{2m-1}}
u(x)(-\Delta_{x})^{m-\frac{1}{2}}u(x)dx
\end{split}
\end{equation*}
and
\begin{equation*}
\begin{split}
\int_{\mathbb{R}_{+}^{2m}}\Delta U\Delta^{m-1}Udxdy
=&(-1)^{m-2}\int_{\mathbb{R}_{+}^{2m}}|\nabla^{m} U|^{2}dxdy.
\end{split}
\end{equation*}
Therefore,
\begin{equation*}
\begin{split}
\int_{\mathbb{R}_{+}^{2m}}|\nabla^{m} U|^{2}dxdy=&\frac{\Gamma(m)\Gamma(\frac{1}{2})}{\Gamma(m-\frac{1}{2})}
\int_{\mathbb{R}^{2m-1}}
u(x)(-\Delta_{x})^{m-\frac{1}{2}}u(x)dx.
\end{split}
\end{equation*}
The desired result follows.
\end{proof}

\section{Fractional Adams inequalities with the exact growth in $W^{\frac{n}{2},2}(\mathbb{R}^n)$}
In this section, we shall prove the fractional Adams inequalities with the exact growth in $W^{\frac{n}{2},2}(\mathbb{R}^n)$ for all $n\geq 2$.
For this purpose, we need the following lemmas.
\medskip

\begin{lemma}\cite{Be1}\label{rad}
For any radial function $u \in {W^{\frac{{n}}{2},2}}\left( {{\mathbb{R}^n}} \right)$, one has
\[\int_{{\mathbb{R}^n}} {{{\left| {{({-\Delta}) ^{\frac{{n}}{4}}}u} \right|}^2}} dx \ge {{\rm{2}}^{n - 2}}{\Gamma ^2}\left( {\frac{n}{2}} \right)
\int_{{\mathbb{R}^n}} {\frac{{{{\left| {\nabla u} \right|}^2}}}{{{{\left| x \right|}^{n-2} }}}} dx.\]
\end{lemma}

\begin{lemma}\label{cont}
For any $u \in W^{\frac{n}{2}, 2}\left(\mathbb{R}^{n}\right)$, one has
$$\int_{\mathbb{R}^{2 m}} \frac{e^{\beta(n,\frac{n}{2})|u|^{2}}-1}{1+|u|^{2}} d x \leq 2\left(e^{\beta(n,\frac{n}{2})}-1-2 \beta(n,\frac{n}{2})\right)\|u\|_{2}^{2}+2 \int_{\mathbb{R}^{2 m}} \frac{e^{\beta(n,\frac{n}{2})\left|u^{\#}\right|^{2}}-1}{1+\left|u^{\#}\right|^{2}} d x,$$
where $u^{\sharp}=\mathcal{F}^{-1}\{(\mathcal{F}u)^{\ast}\}$ denotes the Fourier rearrangement of $u$.
\end{lemma}
\begin{proof}This lemma can be similarly proved as Lemma 2.1 in \cite{N}. We omit the details.
\end{proof}
With these two lemmas, we can give

\begin{proof}[Proof of Theorem \ref{fratru}]
Assume $u\in W^{\frac{n}{2},2}(\mathbb{R}^n)$ with $\int_{\mathbb{R}^n}|(-\Delta)^{\frac{n}{4}} u(x)|^2dx=1$, then by the Fourier rearrangement inequality \cite{Lenza}, there exists some radial function $u^{\sharp} \in {W^{\frac{{n}}{2},2}}\left( {{\mathbb{R}^n}} \right)$ such that
\[\int_{{\mathbb{R}^n}} {{{| {{{\left( { - \Delta } \right)}^{\frac{{n}}{4}}}{u^{\rm{\sharp }}}}|}^2}} dx\leq 1.\]
Combining this and Lemma \ref{cont}, we can assume $u$ is a radial function satisfying $$\int_{{\mathbb{R}^n}} {{{| {{{\left( { - \Delta } \right)}^{\frac{{n}}{4}}}{u}} |}^2}} dx= 1$$ and $u\left( x \right) \to 0$ as $\left| x \right| \to \infty
$.
\vskip0.1cm

In view of Lemma \ref{rad}, we have


\begin{align*}
{\int _{{\mathbb{R}^n}}}{| {{{( - \Delta )}^{\frac{n}{4}}}u} |^2}dx & \ge {2^{n - 2}}{\Gamma ^2}\left( {\frac{n}{2}} \right){\int _{{\mathbb{R}^n}}}\frac{{{{\left| {\nabla u} \right|}^2}}}{{{{\left| x \right|}^{n - 2}}}}dx\\& = {2^{n - 2}}{\Gamma ^2}\left( {\frac{n}{2}} \right)\frac{{{2^{\frac{{n + 2}}{2}}}{\pi ^{\frac{{n - 1}}{2}}}}}{{\left( {n - 2} \right)!!}}\int_0^\infty  {{{\left| {u'} \right|}^2}r} dr\\&  = 2\pi \cdot\frac{{{{\rm{2}}^{n - 3}}{\Gamma ^2}\left( {\frac{n}{2}} \right)}}{\pi }\frac{{{2^{\frac{{n + 2}}{2}}}{\pi ^{\frac{{n - 1}}{2}}}}}{{\left( {n - 2} \right)!!}}\int_0^\infty  {{{\left| {u'} \right|}^2}r} dr\\&=2\pi \left( {{A_n}\int_0^\infty  {{{\left| {u'} \right|}^2}r} dr} \right),
\end{align*}
 where  $$A_n=\frac{{{{\rm{2}}^{n-3}}{\Gamma ^2}\left( {\frac{{n}}{2}} \right)}}{{\pi}}\frac{{{2^{\frac{{n + 2}}{2}}}{\pi ^{\frac{{n - 1}}{2}}}}}{{\left( {n - 2} \right)!!}}.
 $$
  Denote $w\left( s \right) = {\left( {\frac{n}{2}{A_n }} \right)^{1/2}}u\left( {{s^{2/n}}} \right)$, that is \[ u\left( r \right)={\left( {\frac{n}{2}{A_n }} \right)^{ - 1/2}}w\left( {{r^{n/2}}} \right) .\]
Then direct calculation gives
 \begin{align*}
 \int_0^\infty  {{{\left| {w'\left( s \right)} \right|}^2}s} ds &= \int_0^\infty  {{{\left| {{{\left( {\frac{n}{2}{A_n }} \right)}^{1/2}}u'\left( {{s^{2/n}}} \right)\frac{2}{n}{s^{\frac{{2 - n}}{n}}}} \right|}^2}s} ds \\
 & = {A_n }\int_0^\infty  {{{\left| {u'\left( {{s^{2/n}}} \right)} \right|}^2}{s^{2/n}}} d{s^{2/n}} \\
 & {\rm{ = }}{A_n }\int_0^\infty  {{{\left| {u'\left( r \right)} \right|}^2}} rdr.
\end{align*}
Therefore, it follows that
 \[1 \ge 2\pi \left( {{A_n }\int_0^\infty  {{{\left| {u'} \right|}^2}r} dr} \right) = 2\pi \left( {\int_0^\infty  {{{\left| {w'\left( s \right)} \right|}^2}s} ds} \right)\]
and
\begin{align*}
 \int_0^\infty  {{{\left| {w\left( s \right)} \right|}^2}s} ds& = \int_0^\infty  {{{\left| {{{\left( {\frac{n}{2}{A_n }} \right)}^{1/2}}u\left( {{s^{2/n}}} \right)} \right|}^2}s} ds \\
& {\rm{ = }}\frac{{n{A_n }}}{2}\int_0^\infty  {{{\left| {u\left( r \right)} \right|}^2}{r^{n/2}}} d{r^{n/2}} \\
 &  = \frac{{{n^2}{A_n }}}{4}\int_0^\infty  {{{\left| {u\left( r \right)} \right|}^2}{r^{n - 1}}} dr. \\
\end{align*}

Gathering the above estimate, we conclude that
\begin{align*}
 \int_{{\mathbb{R}^n}} {\frac{{ {\exp \left( {{\beta(n,\frac{n}{2})}{u^2}} \right) - 1} }}{{1 + {{\left| u \right|}^2}}}dx} & = {\omega _{n - 1}}\int_0^\infty  {\frac{{ {\exp \left( {{\beta(n,\frac{n}{2})}{{\left( {{{\left( {\frac{n}{2}{A_n }} \right)}^{ - 1/2}}w\left( {{r^{\frac{n}{2}}}} \right)} \right)}^2}} \right) - 1} }}{{1 + {{\left| {{{\left( {\frac{n}{2}{A_n }} \right)}^{ - 1/2}}w\left( {{r^{\frac{n}{2}}}} \right)} \right|}^2}}}{r^{n - 1}}} dr \\
& {\rm{ = }}\frac{{{\omega _{n - 1}}}}{n}\int_0^\infty  {\frac{{\exp \left( {\frac{{2{\beta(n,\frac{n}{2})}}}{{n{A_n }}}{{ {w^2( {{r^{\frac{n}{2}}}})} }}} \right) - 1}}{{1 + {{\left| {{{\left( {\frac{n}{2}{A_n }} \right)}^{ - 1/2}}w\left( {{r^{\frac{n}{2}}}} \right)} \right|}^2}}}} d{r^n} \\
& {\rm{ = }}\frac{{{\omega _{n - 1}}}}{n}\int_0^\infty  {\frac{{\exp \left( {4\pi {{{w^2( {{r^{\frac{n}{2}}}} )} }}} \right) - 1}}{{1 + {{\left( {\frac{n}{2}{A_n }} \right)}^{ - 1}}{{\left| {w\left( {{r^{\frac{n}{2}}}} \right)} \right|}^2}}}} d{r^n} \\
 & \le {A_n }{\omega _{n - 1}}\int_0^\infty  {\frac{{\exp \left( {4\pi {{ {w^2(s)}}}} \right) - 1}}{{1 + {{| {w(s)} |}^2}}}} sds \\
 & \le c\int_{{\mathbb{R}^2}} {w^2{{\left( {|x|} \right)}}} dx  \le c\int_{{\mathbb{R}^n}} {{u^2}} dx, \\
\end{align*}
which accomplishes the proof of (\ref{addadmexa}).
\medskip

Now, we prove the sharpness of $p\ge2$. We adopt the test function sequence used in \cite{Fon}:
\[{\phi _{\varepsilon ,r}}\left( x \right) = \left\{ {\begin{array}{*{20}{c}}
   {{c_{n/2}{\left| x \right|}^{ - n/2}},\ \varepsilon r < \left| x \right| < r}  \\
   {0,\ \ \ \text{otherwise}},  \\
\end{array}} \right.\]
where $c_{n/2}=\frac{1}{{{2^{n/2}}{\pi ^{n/2}}}}
$ and define the function on ${B_r}$ by :
\[\widetilde{{\phi _{\varepsilon ,r}}} = {\phi _{\varepsilon ,r}}\left( x \right) - P_{n - 1}^r{\phi _\varepsilon },\]
where $P_{n - 1}^r{\phi _\varepsilon }$ is the projection on the space of polynomials of degree up to $n-1$ on the ball $B_r$. Obviously,  $\widetilde{{\phi _{\varepsilon ,r}}}$ is orthogonal to every polynomial of order up to $n-1$ on the ball ${B_r}$ and satisfies $\left| {P_{n - 1}^r{\phi _\varepsilon }} \right| \le C{r^{ - \frac{n}{2}}}$.
\vskip0.1cm

By (100), (104) and (107) in \cite{Fon}, we have \[\left\| {\widetilde{{\phi _{\varepsilon ,r}}}} \right\|_2^2 \le  - \frac{{\log {{\left( {\varepsilon r} \right)}^n}}}{{\beta \left( {n,\frac{n}{2}} \right)}} + {b_r} + C,\]
 \begin{equation}
 \label{contr}
 \left\| {{I_{n/2}}\widetilde{{\phi _{\varepsilon ,r}}}} \right\|_2^2 \le C{r^n}\end{equation}
and
\[{I_{n/2}}\widetilde{{\phi _{\varepsilon ,r}}}\left( x \right) \ge  - \frac{{\log {{\left( {\varepsilon r} \right)}^n}}}{{\beta \left( {n,\frac{n}{2}} \right)}} + {b_r} - C\]
on the ball ${B_{\varepsilon r/2}}$, where ${b_r} \le C\log r$.

On the other hand, since \begin{align*}
 {\left| {\widetilde{{\phi _{\varepsilon ,r}}}} \right|^2}& = {| {{\phi _{\varepsilon ,r}}\left( x \right) - P_{n - 1}^r{\phi _\varepsilon }} |^2} \\
  &\ge {C_1}{\left| {{\phi _{\varepsilon ,r}}\left( x \right)} \right|^2} - {C_2}{\left| {P_{n - 1}^r{\phi _\varepsilon }} \right|^2} \\
  &\ge {C_1}{\left| {{\phi _{\varepsilon ,r}}\left( x \right)} \right|^2} - {C_2}{r^{ - n/2}}, \\
\end{align*}
and thus it follows that \begin{equation}\label{adds}\left\| {\widetilde{{\phi _{\varepsilon ,r}}}} \right\|_2^2 \ge  - {C_1}\log {\left( {\varepsilon r} \right)^n} + {C_2}.\end{equation}

Set ${\psi _{\varepsilon ,r}} = \frac{{{I_{n/2}}\widetilde{{\phi _{\varepsilon ,r}}}}}{{{{\left\| {\widetilde{{\phi _{\varepsilon ,r}}}} \right\|}_2}}}$, then on the ball ${B_{\varepsilon r/2}}$,
\begin{align*}
 {\left( {{\psi _{\varepsilon ,r}}\left( x \right)} \right)^2} &\ge {\frac{{\left( { - \frac{{\log {{\left( {\varepsilon r} \right)}^n}}}{{\beta \left( {n,\frac{n}{2}} \right)}} + {b_r} - C} \right)}}{{ - \frac{{\log {{\left( {\varepsilon r} \right)}^n}}}{{\beta \left( {n,\frac{n}{2}} \right)}} + {b_r} + C}}^2} \\
 & \ge  - \frac{{\log {{\left( {\varepsilon r} \right)}^n}}}{{\beta \left( {n,\frac{n}{2}} \right)}} + {b_r}\left( {1 - \frac{C}{{\log \frac{1}{{{\varepsilon ^n}}}}}} \right) - C,\\
\end{align*}
which gives
\begin{align*}
 & \mathop {\sup }\limits_{\left\| {{{\left( { - \Delta } \right)}^{n/4}}u} \right\|_2 \le 1} \int_{{\mathbb{R}^n}} {\frac{{\exp \left( {\beta \left( {n,\frac{n}{2}} \right){u^2}} \right) - 1}}{{{{\left( {1 + u} \right)}^p}}}dx} \\ &\ge \int_{{B_{r\varepsilon /2}}} {\frac{{\exp \left( {\beta \left( {n,\frac{n}{2}} \right){{\left( {{\psi _{\varepsilon ,r}}\left( x \right)} \right)}^2}} \right)}}{{{{\left( {1 + {\psi _{\varepsilon ,r}}\left( x \right)} \right)}^p}}}dx}  \\
  &\ge \int_{{B_{r\varepsilon /2}}} {\frac{{\exp \left( {\beta \left( {n,\frac{n}{2}} \right)\left( { - \frac{{\log {{\left( {\varepsilon r} \right)}^n}}}{{\beta \left( {n,\frac{n}{2}} \right)}} + {b_r}\left( {1 - \frac{C}{{\log \frac{1}{{{\varepsilon ^n}}}}}} \right) - C} \right)} \right)}}{{{{\left( { - \frac{{\log {{\left( {\varepsilon r} \right)}^n}}}{{\beta \left( {n,\frac{n}{2}} \right)}} + {b_r}\left( {1 - \frac{C}{{\log \frac{1}{{{\varepsilon ^n}}}}}} \right) - C} \right)}^{\frac{p}{2}}}}}dx}  \\
 & \ge \frac{C}{{{{\left( {\log \frac{1}{{{{\left( {\varepsilon r} \right)}^n}}}} \right)}^{\frac{p}{2}}}}}. \\
\end{align*}

On the other hand, from (\ref{contr}) and (\ref{adds}), we see that \[\int_{{\mathbb{R}^n}} {{{\left( {{\psi _{\varepsilon ,r}}\left( x \right)} \right)}^2}dx}  = \frac{{\left\| {{I_{n/2}}\widetilde{{\phi _{\varepsilon ,r}}}} \right\|_2^2}}{{\left\| {\widetilde{{\phi _{\varepsilon ,r}}}} \right\|_2^2}} \ge \frac{{C{r^n}}}{{ - \log {{\left( {\varepsilon r} \right)}^n} + C}}.\]

Combining the above estimate, we conclude that if $p<2$,
\begin{align*}
 \mathop {\sup }\limits_{{{\left\| {{{\left( { - \Delta } \right)}^{  n/4}}u} \right\|}_2} \le 1} \frac{1}{{\left\| u \right\|_2^2}}\int_{{\mathbb{R}^n}} {\frac{{\exp \left( {\beta {u^2}} \right) - 1}}{{{{\left( {1 + u} \right)}^p}}}dx}  &\ge \frac{1}{{\left\| {{\psi _{\varepsilon ,r}}\left( x \right)} \right\|_2^2}}\int_{{B_{r\varepsilon /2}}} {\frac{{\exp \left( {\beta {{\left( {{\psi _{\varepsilon ,r}}\left( x \right)} \right)}^2}} \right)}}{{{{\left( {1 + {\psi _{\varepsilon ,r}}\left( x \right)} \right)}^p}}}dx}  \\
&  \ge C\frac{{\log \frac{1}{{{{\left( {\varepsilon r} \right)}^n}}}}}{{{r^n}{{\left( {\log \frac{1}{{{{\left( {\varepsilon r} \right)}^n}}}} \right)}^{\frac{p}{2}}}}} \\
 & = \frac{C}{{{r^n}}}{\left( {\log \frac{1}{{{{\left( {\varepsilon r} \right)}^n}}}} \right)^{1 - \frac{p}{2}}} \to \infty  \\
 \end{align*}
as $ \varepsilon  \to 0$, which accomplishes the proof of Theorem \ref{fratru}.

\end{proof}

\end{document}